\documentclass[a4paper,11pt]{article}

\usepackage{stmaryrd}
\usepackage{amsfonts}
\usepackage{bbm}

\usepackage[all]{xy}

\usepackage{amscd}
\usepackage{mathrsfs}
\usepackage{latexsym,amssymb,amsmath,amscd,amscd,amsthm,amsxtra}
\usepackage[dvips]{graphicx}
\usepackage[utf8]{inputenc}
\usepackage[T1]{fontenc}
\usepackage{enumerate}
\usepackage{lmodern}
\usepackage{amssymb}
\usepackage{nicefrac,mathtools,enumitem}
\usepackage{microtype}
\usepackage{amsfonts}
\usepackage{amssymb}
\usepackage{mathrsfs}
\usepackage{amsmath}
\usepackage{amsthm}
\usepackage{enumerate}
\usepackage{indentfirst}
\usepackage{amsfonts}
\usepackage{amssymb}
\usepackage{mathrsfs}
\usepackage{amsmath}
\usepackage{amsthm}
\usepackage{enumerate}
\usepackage{cite}
\usepackage{mathrsfs}
\usepackage{kpfonts}
\usepackage{geometry}
\usepackage{ifpdf}
\ifpdf
  \usepackage[colorlinks=true,linkcolor=blue,citecolor=red, final,backref=page,hyperindex]{hyperref}
\else
  \usepackage[colorlinks,final,backref=page,hyperindex,hypertex]{hyperref}
\fi

\allowdisplaybreaks

\makeatletter
\newcommand{\subjclass}[2][2020]{%
  \let\@oldtitle\@title%
  \gdef\@title{\@oldtitle\footnotetext{#1 \emph{Mathematics subject classification}: #2}}%
}
\newcommand{\keywords}[1]{%
  \let\@@oldtitle\@title%
  \gdef\@title{\@@oldtitle\footnotetext{\emph{Keywords}: #1}}%
}
\makeatother

\allowdisplaybreaks
\newtheorem{defn}{Definition}[section]
\newtheorem{thm}[defn]{Theorem}
\newtheorem{lem}[defn]{Lemma}
\newtheorem{prop}[defn]{Proposition}
\newtheorem{prop-defn}[defn]{Proposition-Definition}
\newtheorem{cor}[defn]{Corollary}
\newtheorem{ex}[defn]{Example}
\newtheorem{re}[defn]{Remark}
\newcommand{\bdefn}{\begin{defn}}
\newcommand{\edefn}{\end{defn}}
\newcommand{\bthm}{\begin{thm}}
\newcommand{\ethm}{\end{thm}}
\newcommand{\blem}{\begin{lem}}
\newcommand{\elem}{\end{lem}}
\newcommand{\bprop}{\begin{prop}}
\newcommand{\eprop}{\end{prop}}
\newcommand{\bcor}{\begin{cor}}
\newcommand{\ecor}{\end{cor}}
\newcommand{\beg}{\begin{eg}}
\newcommand{\eeg}{\end{eg}}
\newcommand{\bre}{\begin{re}}
\newcommand{\ere}{\end{re}}
\newcommand{\bpf}{\begin{proof}}
\newcommand{\epf}{\end{proof}}

\newcommand{\benu}{\begin{enumerate}}
\newcommand{\eenu}{\end{enumerate}}
\newcommand{\bc}{\begin{center}}
\newcommand{\ec}{\end{center}}
\newcommand{\bea}{\begin{eqnarray}}
\newcommand{\eea}{\end{eqnarray}}
\newcommand{\ba}{\begin{align*}}
\newcommand{\ea}{\end{align*}}
\newcommand{\Bea}{\begin{eqnarray*}}
\newcommand{\Eea}{\end{eqnarray*}}
\newcommand{\beq}{\begin{equation}}
\newcommand{\eeq}{\end{equation}}
\newcommand{\Beq}{\begin{equation*}}
\newcommand{\Eeq}{\end{equation*}}
\newcommand{\bspl}{\begin{split}}
\newcommand{\espl}{\end{split}}

\numberwithin{equation}{section}

\begin{document}
\date{}
\title{\bf  Homogeneous $\mathcal{O}$-operators on BiHom-Jordan superalgebras and BiHom-pre-Jordan superalgebras structures }
\author{ Sami Mabrouk and Othmen Ncib}
\author{\normalsize \bf  Othmen Ncib\small{$^{1}$} \footnote{  E-mail: othmenncib@yahoo.fr } }
\date{{\small{$^{1}$}   Faculty of Sciences, University of Gafsa,   BP
2100, Gafsa, Tunisia }}

\maketitle

\begin{center}
    \emph{To the memory of Salem OMRI}
    \vspace{0.5cm}
\end{center}
\begin{abstract}

\end{abstract}

In this paper, we recall the notion of BiHom- generalization of Jordan superalgebras called BiHom-Jordan superalgebras in which, we investigate their representation and associated dual representation theory. We also introduce the notion of homogeneous $\mathcal{O}$-operators of BiHom-Jordan superalgebra associated to a given representation. Additionnaly we define the BiHom-pre-Jordan superalgebras structures, in which we study the relationship between BiHom-pre-Jordan superalgebras and BiHom-Jordan superalgebras via a homogeneous $\mathcal{O}$-operators. Some other related results are considered.\\

{\bf Keywords:} BiHom-Jordan superalgebra, BiHom-pre-Jordan superalgebra, representation, $\mathcal{O}$-Operator.

\textbf{Mathematics Subject Classification}: 17A15, 17A60, 17B38, 17B60, 17D30, 17C10
\tableofcontents
\section{Introduction}
A Jordan superalgebra is a $\mathbb Z_2$-graded vector space $\mathcal{J}$ equipped with an even super-commutative bilinear map (i.e. $x\cdot y =(-1)^{|x||y|} y \cdot x$) that satisfies the Jordan super-identity:
\begin{align*}
  \sum_{x,y,u}(-1)^{|t|(|x|+|z|)}as(x\cdot y,z,t)=0.
\end{align*}
where $x,y,z,t\in \mathcal{H}(\mathcal{J})$ and $\displaystyle\sum_{x,y,t}$ denoted the cyclic sum over $(x,y,z)$ and $as(x,y,z)=(x\cdot y)\cdot z-x\cdot(y\cdot z)$ for any $x,y,z\in \mathcal{J}$. The reader is referred to \cite{Cantarini-Kac} and  \cite{shest-jordan-super} for discussions about the important role of Jordan superalgebras in physics, especially in quantum
mechanics. For interested readers, please refer also to the literature \cite{Shest-Gonzalez,Martinez-Zelmanov}. These algebraic structures were first studied in \cite{VG} by classifying finite-dimensional simple
Jordan superalgebras over an algebraically closed field of characteristic zeros and have been rapidly developed \cite{VG,RZ}. In \cite{VG}, super Jordan Yang-Baxter Equation (super JYBE) in Jordan superalgebras and super $\mathcal{O}$-operators are introduced.

In the past few years, many scholars have conducted indepth research on twisted algebras (also known as Hom-algebras or Bihom-algebras) excited by $q$-deformation
of vector field algebras. We can find more related results in \cite{Gohr,G-M-M-P,G-Z-W,L-C-M,Makhlouf1,Wang-Guo,Chtioui1,M-N-S}. We noticed recently that Hom-Jordan superalgebras and BiHom-Jordan superalgebras were studied in \cite{KFA,Hou-Chen}. Some other results studied on Hom-Jordan (super)-algebras and BiHom-Jordan (super)-algebras can be found in \cite{Chtioui2,DNC,M-N-S,wang,Yau}.\\

Representations theory of different algebraic structures is an important subject of study in algebra and diverse areas. They appear in many fields of mathematics and physics. In particular, they appear in deformation and cohomology theory among other areas. The notion of representation introduced for Hom–Lie algebras in\cite{Y-Sheng1} (see also\cite{Benayadi-Makhlouf}), it's also introduced for Jordan superalgebras and Hom-Jordan superalgebras in \cite{M-N-S,RZ}.\\

In $1960$, G. Baxter \cite{Baxter} first introduced the notion
of Rota-Baxter operators for associative algebras. The Rota-Baxter operators have several
applications in probability \cite{Baxter}, combinatorics \cite{Cartier,Guo-Keigher,G-C-Rota}, and quantum field theory \cite{Connes-Kreimer}. In the
$1980s$, the notion of Rota-Baxter operator of weight zero was introduced in terms of the classical
Yang-Baxter equation for Lie algebras (see \cite{Guo} for more details). Later on, B. A. Kupershmidt
\cite{B-Kupershmidt} defined the notion of $\mathcal{O}$-operators as generalized Rota-Baxter operators to understand classical Yang-Baxter equations and related integrable systems. Several results associated to this notion are studied for several algebraic structures ( see \cite{Chtioui2,DNC,Harrathi,M-N-S,wang} for more details). In the super case, a super $\mathcal{O}$-operator on a superalgebra $(\mathfrak g,\cdot)$ with respect to a representation of $\mathfrak g$ on a $\mathbb Z_2$-graded vector space $V$ is defined as an even linear map $T:V\to \mathfrak g$ satisfying certain condition. In \cite{Bai-Guo-Zhang}, the authors introduced the notion of homogeneous $\mathcal{O}$-operators on Lie superalgebras in which they gave some results associated to the even and odd $\mathcal{O}$-operators. In our paper, we extend this notion to the BiHom-Jordan superalgebras in which we introduce the notion of homogeneous $\mathcal{O}$-operators on BiHom-Jordan superalgebras with respect to a given representation and  we study the relationship between BiHom-pre-Jordan super-algebra and BiHom-Jordan superalgebra via a homogeneous $\mathcal{O}$-operators.  \\

\textbf{Layout of the paper}. Here is an outline of the paper.\\
In Section \ref{Sec2}, we recall relevant definitions
and properties of BiHom-Jordan superalgebras in which we introduce the representation theory of these algebraic structures and obtain some related results. In Section \ref{Sec3}, we study the notion of dual representation of a given representation of BiHom-Jordan superalgebra without any condition and we give some related results. Section \ref{Sec4} discusses the notions of BiHom-pre-Jordan superalgebras and the homogeneous $\mathcal{O}$-operators on BiHom-Jordan superalgebras associated with a given representation in which we study the relationship between BiHom-pre-Jordan superalgebras and BiHom-Jordan superalgebras via a homogeneous $\mathcal{O}$-operator. We also give some results concerning the notions of the parity reverse of a representation and $\mathcal{O}$-operators duality introduced in \cite{Bai-Guo-Zhang}, in which we obtain
a natural one-one correspondence between $\mathcal{O}$-operators with any given parity associated to a
representation of a BiHom-Jordan superalgebra and $\mathcal{O}$-operators with the opposite parity associated to the parity reverse representation which is called $\mathcal{O}$-operators duality. 

Recall that, in general a superalgebra means a $\mathbb{Z}_2$-graded algebra, that is an algebra $\mathcal{A}$ which may be written as a direct sum of subspaces $\mathcal{A}=\mathcal{A}_{0}\oplus\mathcal{A}_{1}$ subject to the relation $\mathcal{A}_{i}\mathcal{A}_{j}\subseteq \mathcal{A}_{i+j}$.
The subspaces $\mathcal{A}_{0}$ and $\mathcal{A}_{1}$ are called the even and the odd parts of the superalgebra $\mathcal{A}$. Define the $\mathbb{Z}_2$-graded vector space $\mathcal{A}^{op}=\mathcal{A}_{1}\oplus\mathcal{A}_{0}$.
Throughout this paper, $\mathbb{K}$ denotes an algebraically closed field of characteristic 0. All algebras and vector spaces are considered on $\mathbb{K}$. The parity of the homogeneous element $x$ is denoted by $|x|$.
 \section{Preliminaries and Basics on BiHom-Jordan superalgebras}\label{Sec2}

 In this section, we recall the notion of BiHom-Jordan superalgebras introduced in \cite{Hou-Chen}, and we introduce the associated BiHom-super-bimodules and represntations for this we develop some helpful results that we will use later.
 \begin{defn}\cite{Hou-Chen}
  A BiHom-Jordan superalgebra is a quadruple $(\mathfrak J,\cdot,\alpha,\beta)$ consisting of a $\mathbb Z_2$-graded vector space $\mathfrak J$, an even bilinear map $\cdot:\mathfrak J\times\mathfrak J\to\mathfrak J$ and two even commuting linear maps $\alpha,\beta:\mathfrak J\to\mathfrak J$ such that for any $x,y,z,t\in\mathcal{H}(\mathfrak J)$, the following conditions hold
 \begin{eqnarray}
&&\beta(x)\cdot\alpha(y)=(-1)^{|x||y|}\beta(y)\cdot\alpha(x),\;~~~~~~(\text{BiHom-super-symmetry condition})\label{J-BiH-skewsym-cond}\\
&&\displaystyle\circlearrowleft_{x,y,t}(-1)^{|t|(|x|+|z|)}ass_{\alpha,\beta}(\beta^2(x)\cdot\alpha\beta(y),\alpha^2\beta(z),\alpha^3(t))=0,\;(\text{Bihom-Jordan super-identity})\label{BiH-Jor-identity}
 \end{eqnarray} 
 where $\displaystyle\circlearrowleft_{x,y,t}$ denotes the cyclic sum over $(x,y,z)$ and $ass_{\alpha,\beta}(x,y,z)=(x\cdot y)\cdot\beta(z)-\alpha(x)\cdot(y\cdot z)$. A BiHom-Jordan superalgebra is called multiplicative if $\alpha$ and $\beta$ are a superalgebras morphisms (i.e. for any $x,y\in\mathfrak{J}$, we have $\alpha(x\cdot y)=\alpha(x)\cdot\alpha(y)$ and $\beta(x\cdot y)=\beta(x)\cdot\beta(y)$ ) and called regular if $\alpha$ and $\beta$ are a superalgebra automorphisms.
 \end{defn}
 \begin{re}\
 \begin{enumerate}
\item If $\alpha=\beta=Id_{\mathfrak{J}}$, we recover the Jordan superalgebras structures (\cite{Cantarini-Kac}).
\item If $\alpha=\beta$ and $\alpha$ is surjective, we recover the Hom-Jordan superalgebras structures (\cite{KFA}). 
 \end{enumerate}    
 \end{re}
 
 \begin{thm}\cite{Hou-Chen}
Let $(\mathfrak J,\cdot)$ be a Jordan superalgebra and $\alpha,\beta$  be two even commuting superalgebra endomorphisms of $\mathfrak J$. Then, $\mathfrak J_{\alpha,\beta}=(\mathfrak J,\cdot_{\alpha,\beta},\alpha,\beta)$ where $\cdot_{\alpha,\beta}$ is defined by 
$$x\cdot_{\alpha,\beta}y=\alpha(x)\cdot\beta(y),\;\forall x,y\in\mathfrak J,$$
is a BiHom-Jordan superalgebra.
 \end{thm}
 \begin{prop}
Let $(\mathfrak J,\cdot,\alpha,\beta)$ be a multiplicative regular BiHom-Jordan superalgebra. Then $(\mathfrak J,\cdot')$, where $\cdot':\mathfrak J\to\mathfrak J$ is defined by
$$x\cdot' y=\alpha^{-1}(x)\cdot\beta^{-1}(y),\;\forall x,y\in\mathfrak J,$$
is a Jordan superalgebra.
 \end{prop}
\begin{proof}
 Let $x,y\in\mathcal{H}(\mathfrak J)$, then by using the BiHom-symmetry of $"\cdot"$, we have
 \begin{align*}
x\cdot'y=\alpha^{-1}(x)\cdot\beta^{-1}(y)&=\beta(\alpha^{-1}\beta^{-1}(x))\cdot\alpha(\alpha^{-1}\beta^{-1}(y))\\&=(-1)^{|x||y|}\beta(\alpha^{-1}\beta^{-1}(y))\cdot\alpha(\alpha^{-1}\beta^{-1}(x))\\&=(-1)^{|x||y|}\alpha^{-1}y\cdot \beta^{-1}(x)\\&=(-1)^{|x||y|}y\cdot'x.
 \end{align*}
 Then $"\cdot'"$ is super-symmetric.\\
 Let $x,y,z,t\in\mathfrak J$, then by using the fact that, $(\mathfrak J,\cdot,\alpha,\beta)$ is a BiHom-Jordan superalgebra, we have 
 \begin{align*}
 ass(x\cdot'y,z,t)&=((x\cdot'y)\cdot'z)\cdot't-(x\cdot'y)\cdot'(z\cdot't)\\&=\alpha^{-1}\big((x\cdot'y)\cdot'z\big)\cdot\beta^{-1}(t)-\alpha^{-1}\big(\alpha^{-1}(x)\cdot\beta^{-1}(y)\big)\cdot\beta^{-1}\big(\alpha^{-1}(z)\cdot\beta^{-1}(t)\big)\\&=\big((\alpha^{-3}(x)\cdot\alpha^{-2}\beta^{-1}(y))\cdot\alpha^{-1}\beta^{-1}(z)\big)\cdot\beta^{-1}(t)-\big(\alpha^{-2}(x)\cdot\alpha^{-1}\beta^{-1}(y)\big)\cdot\big(\alpha^{-1}\beta^{-1}(z)\cdot\beta^{-2}(t)\big).    
 \end{align*}
 We set $X=\alpha^{-3}\beta^{-2}(x),\;Y=\alpha^{-3}\beta^{-2}(y),\;Z=\alpha^{-3}\beta^{-2}(z)$ and $T=\alpha^{-3}\beta^{-2}(t)$, then we have 
 \begin{align*}
ass(x\cdot'y,z,t)&=\big((\beta^2(X)\cdot\alpha\beta(Y))\cdot\alpha^2\beta(Z)\big)\alpha^2\beta(T)-\big(\alpha\beta^2(x)\cdot\alpha^2\beta(Y)\big)\cdot\big(\alpha^2\beta(Z)\cdot\alpha^3(T)\big)\\&=ass_{\alpha,\beta}\big(\beta^2(X)\cdot\alpha\beta(Y),\alpha^2\beta(Z),\alpha^3(T)\big)\\&=0,
 \end{align*}
then, the Jordan identity is satisfied by $\cdot'$ which gives the result. 
\end{proof}

More generally, we have the following result 
\begin{prop}
Let $(\mathfrak J,\cdot,\alpha,\beta)$ be a BiHom-Jordan superalgebra and. Then $(\mathfrak J,\cdot')$, where $\cdot':\mathfrak J\to\mathfrak J$ is defined by
$$x\cdot' y=\alpha^{s-1}\beta^t(x)\cdot\alpha^s\beta^{t-1}(y),\;s,t\in\mathbb Z,$$
is a Jordan superalgebra.    
\end{prop}
\begin{defn}
Let $(\mathfrak J,\cdot,\alpha,\beta)$ be a BiHom-Jordan superalgebra and $V$ be a $\mathbb Z_2$-graded vectore space. Let $\mathfrak l,\mathfrak r:\mathfrak J\to gl(V)$ be two even linear maps and $\alpha_V,\beta_V$ two even commuting endomorphisms of $V$. The tuple $(V,\mathfrak l,\mathfrak r,\alpha_V,\beta_V)$ is called  BiHom-super-bimodule of $\mathfrak J$, if for any $x,y,z\in\mathcal{H}(\mathfrak J)$ and $u\in\mathcal{H}(V)$, the following condition hold:
\begin{eqnarray}
&&\mathfrak l(\beta(x))\alpha_V(u)=\mathfrak r(\alpha(x))\beta_V(u),\label{con-bimod-J-1}\\
&&\mathfrak l\big((\beta^2(x)\cdot\alpha\beta(y))\cdot\alpha^2\beta(z)\big)\alpha_V^3\beta_V-\mathfrak l\big(\alpha\beta^2(x)\cdot\alpha^2\beta(y)\big)\mathfrak l(\alpha^2\beta(z))\alpha_V^3\label{con-bimod-J-4}\\
&&\quad=(-1)^{|z|(|x|+|y|)}\mathfrak r(\alpha^2\beta(z)\cdot\alpha^3(x))\alpha_V\mathfrak l(\beta^2(y))\alpha_V\beta_V-(-1)^{|y||z|}\mathfrak r(\alpha^3\beta(x))\mathfrak r(\alpha^2\beta(z))\mathfrak l(\beta^2(y))\alpha_V\beta_V\nonumber\\
&&\quad+(-1)^{|z|(|x|+|y|)}\mathfrak r(\alpha^2\beta(z)\cdot\alpha^3(y))\alpha_V\mathfrak r(\alpha\beta(x))\beta_V^2-(-1)^{|x|(|y|+|z|)}\mathfrak r(\alpha^3\beta(y))\mathfrak r(\alpha^2\beta(z))\mathfrak r(\alpha\beta(x))\beta_V^2,\nonumber\\
&&\mathfrak r(\alpha^3\beta(x))\mathfrak l(\beta^2(y)\cdot\alpha\beta(z))\alpha_V^2\beta_V-\mathfrak l(\alpha\beta^2(y)\cdot\alpha^2\beta(z))\mathfrak r(\alpha^3(x))\alpha_V^2\beta_V\label{con-bimod-J-5}\\
&&\quad=(-1)^{|z|(|x|+|y|)}\mathfrak l\big(\alpha(\beta^2(z)\cdot\alpha\beta(x))\big)\mathfrak r(\alpha^3(y))\alpha_V^2\beta_V-(-1)^{|x|(|y|+|z|)}\mathfrak r(\alpha^3\beta(y))\mathfrak l(\beta^2(z)\cdot\alpha\beta(x))\alpha_V^2\beta_V\nonumber\\
&&\quad+\mathfrak l\big(\alpha(\beta^2(x)\cdot\alpha\beta(y))\big)r(\alpha^3(z))\alpha_V^2\beta_V-(-1)^{|z|(|x|+|y|)}\mathfrak r(\alpha^3\beta(z))\mathfrak l(\beta^2(x)\cdot\alpha\beta(y))\alpha_V^2\beta_V.\nonumber
\end{eqnarray}
\end{defn}

\begin{thm}\label{pr-semi-direct-BiH-Jord-Bim}
Let $(\mathfrak J,\cdot,\alpha,\beta)$ be a BiHom-Jordan superalgebra and $V$ be a $\mathbb Z_2$-graded vector space. Let $\alpha_V,\beta_V:\mathfrak J\to\mathfrak J$ be two even commuting linear maps and $\mathfrak l,\mathfrak r:\mathfrak J\to \mathfrak{gl}(V)$ be two even linear maps. Then, $(V,\mathfrak l,\mathfrak r,\alpha_V,\beta_V)$ is a BiHom-super-bimodule of $\mathfrak J$ if and only if $(\mathfrak J\oplus V,\cdot_{\mathfrak J\oplus V},\alpha\oplus\alpha_V,\beta\oplus\beta_V)$ is a BiHom-Jordan superalgebra where
\begin{eqnarray}
&&\alpha\oplus\alpha_V(x+u)=\alpha(x)+\alpha_V(u),\label{str-pr-semi-dir1}\\
&&\beta\oplus\beta_V(x+u)=\beta(x)+\beta_V(u),\label{str-pr-semi-dir2}\\
&&(x+u)\cdot_{\mathfrak J\oplus V}(y+v)=x\cdot y+\mathfrak l(x)v+(-1)^{|u||y|}\mathfrak r(y)u,\label{str-pr-semi-dir3}
\end{eqnarray}
for any $x,y\in\mathcal{H}(\mathfrak J)$ and $u,v\in\mathcal{H}(V)$.
\end{thm}
\begin{proof}
 Let $x,y\in\mathcal{H}(\mathfrak J)$ and $u,v\in\mathcal{H}(V)$, then, the binary operation $\cdot_{\mathfrak J\oplus V}$ satisfying condition \eqref{J-BiH-skewsym-cond} if and only if 
$$\beta\oplus\beta_V(x+u)\cdot_{\mathfrak J\oplus V}\alpha\oplus\alpha_V(y+v)=\beta\oplus\beta_V(y+v)\cdot_{\mathfrak J\oplus V}\alpha\oplus\alpha_V(x+u),$$
which implies that $$\beta(x)\cdot\alpha(y)+\mathfrak l(\beta(x))\alpha_V(v)+(-1)^{|u||y|}\mathfrak r(\alpha(y))\beta_V(u)=(-1)^{|x||y|}\beta(y)\cdot\alpha(x)+(-1)^{|u||y|}\mathfrak l(\beta(y))\alpha_V(u)+\mathfrak r(\alpha(x))\beta_V(v),$$ which gives by a direct computation that $l(\beta(x))\alpha_V(u)=\mathfrak r(\alpha(x))\beta_V(u)$ for any $x\in\mathcal{H}(\mathfrak J)$ and $u\in\mathcal{H}(V)$.\\
Let $x,y,z,t\in\mathcal{H}(\mathfrak J)$ and $u,v,w,\theta\in\mathcal{H}(V)$, then we have
\small{\begin{eqnarray*}
&&ass_{\mathfrak J\oplus V}\big((\beta\oplus\beta_V)^2(x+u)\cdot_{\mathfrak J\oplus V}(\alpha\oplus\alpha_V)(\beta\oplus\beta_V)(y+v),(\alpha\oplus\alpha_V)^2(\beta\oplus\beta_V)(z+w),(\alpha\oplus\alpha_V)^3(t+\theta)\big)\\&&=ass_{\mathfrak J\oplus V}\big(\beta^2(x)\cdot\alpha\beta(y)+\mathfrak l(\beta^2(x))\alpha_V\beta_V(v)+(-1)^{|x||y|}\mathfrak r(\alpha\beta(y))\beta_V^2(u),\alpha^2\beta(z)+\alpha_V^2\beta_V(w),\alpha^3(t)+\alpha_V^3(\theta)\big)\\&&=\Big(\big(\beta^2(x)\cdot\alpha\beta(y)+\mathfrak l(\beta^2(x))\alpha_V\beta_V(v)+(-1)^{|x||y|}\mathfrak r(\alpha\beta(y))\beta_V^2(u)\big)\cdot_{\mathfrak J\oplus V}\big(\alpha^2\beta(z)+\alpha_V^2\beta_V(w)\big)\Big)\cdot_{\mathfrak J\oplus V}\big(\alpha^3\beta(t)+\alpha_V^3\beta_V(\theta)\big)\\&&-\Big(\alpha(\beta^2(x)\cdot\alpha\beta(y))+\alpha_V\mathfrak l(\beta^2(x))\alpha_V\beta_V(v)+(-1)^{|x||y|}\alpha_V\mathfrak r(\alpha\beta(y))\beta_V^2(u)\Big)\cdot_{\mathfrak J\oplus V}\Big(\big(\alpha^2\beta(z)+\alpha_V^2\beta_V(w)\big)\cdot_{\mathfrak J\oplus V}(\alpha^3(t)+\alpha_V^3(\theta))\Big)\\&&=\big(\big(\beta^2(x)\cdot\alpha\beta(y)\big)\cdot\alpha^2\beta(z)\big)\cdot(\alpha^3\beta(t))-\alpha(\beta^2(x)\cdot\alpha\beta(y))\cdot\big(\alpha^2\beta(z)\cdot\alpha^3(t)\big)+\mathfrak l\big(\big(\beta^2(x)\cdot\alpha\beta(y)\big)\cdot\alpha^2\beta(z)\big)\alpha_V^3\beta_V(\theta)\\&&+(-1)^{|t|(|x|+|y|+|z|)}\mathfrak r\big(\alpha^3\beta(t)\big)\Big(\mathfrak l(\beta^2(x)\cdot\alpha\beta(y))\alpha_V^2\beta_V(w)+(-1)^{|z|(|x|+|y|)}\mathfrak r(\alpha^2\beta(z))\big(\mathfrak l(\beta^2(x))\alpha_V\beta_V(v)+(-1)^{|x||y|}\mathfrak r(\alpha\beta(y))\beta_V^2(u)\big)\Big)\\&&-\mathfrak l(\alpha(\beta^2(x)\cdot\alpha\beta(y)))\big(\mathfrak l(\alpha^2\beta(z))\alpha_V^3(\theta)+(-1)^{|z||t|}\mathfrak r(\alpha^3(t))\alpha_V^2\beta_V(w)\big)\\&&-(-1)^{(|x|+|y|)(|z|+|t|)}\mathfrak r\big(\alpha^2\beta(z)\cdot\alpha^3(t)\big)\big(\alpha_V\mathfrak l(\beta^2(x))\alpha_V\beta_V(v)+(-1)^{|x||y|}\alpha_V\mathfrak r(\alpha\beta(y))\beta_V^2(u)\big).
\end{eqnarray*}}
By the same way, we hve
\begin{eqnarray*}
&&ass_{\mathfrak J\oplus V}\big((\beta\oplus\beta_V)^2(y+v)\cdot_{\mathfrak J\oplus V}(\alpha\oplus\alpha_V)(\beta\oplus\beta_V)(t+\theta),(\alpha\oplus\alpha_V)^2(\beta\oplus\beta_V)(z+w),(\alpha\oplus\alpha_V)^3(x+u)\big)\\&&= \big(\big(\beta^2(y)\cdot\alpha\beta(t)\big)\cdot\alpha^2\beta(z)\big)\cdot(\alpha^3\beta(x))-\alpha(\beta^2(y)\cdot\alpha\beta(t))\cdot\big(\alpha^2\beta(z)\cdot\alpha^3(x)\big)+\mathfrak l\big(\big(\beta^2(y)\cdot\alpha\beta(t)\big)\cdot\alpha^2\beta(z)\big)\alpha_V^3\beta_V(u)\\&&+(-1)^{|x|(|y|+|z|+|t|)}\mathfrak r\big(\alpha^3\beta(x)\big)\Big(\mathfrak l(\beta^2(y)\cdot\alpha\beta(t))\alpha_V^2\beta_V(w)+(-1)^{|z|(|y|+|t|)}\mathfrak r(\alpha^2\beta(z))\big(\mathfrak l(\beta^2(y))\alpha_V\beta_V(\theta)+(-1)^{|y||t|}\mathfrak r(\alpha\beta(t))\beta_V^2(v)\big)\Big)\\&&-\mathfrak l(\alpha(\beta^2(y)\cdot\alpha\beta(t)))\big(\mathfrak l(\alpha^2\beta(z))\alpha_V^3(u)+(-1)^{|x||z|}\mathfrak r(\alpha^3(x))\alpha_V^2\beta_V(w)\big)\\&&-(-1)^{(|x|+|z|)(|y|+|t|)}\mathfrak r\big(\alpha^2\beta(z)\cdot\alpha^3(x)\big)\big(\alpha_V\mathfrak l(\beta^2(y))\alpha_V\beta_V(\theta)+(-1)^{|y||t|}\alpha_V\mathfrak r(\alpha\beta(t))\beta_V^2(v)\big),
\end{eqnarray*}
and
\begin{eqnarray*}
&&ass_{\mathfrak J\oplus V}\big((\beta\oplus\beta_V)^2(t+\theta)\cdot_{\mathfrak J\oplus V}(\alpha\oplus\alpha_V)(\beta\oplus\beta_V)(x+u),(\alpha\oplus\alpha_V)^2(\beta\oplus\beta_V)(z+w),(\alpha\oplus\alpha_V)^3(y+v)\big)\\&&= \big(\big(\beta^2(t)\cdot\alpha\beta(x)\big)\cdot\alpha^2\beta(z)\big)\cdot(\alpha^3\beta(y))-\alpha(\beta^2(t)\cdot\alpha\beta(x))\cdot\big(\alpha^2\beta(z)\cdot\alpha^3(y)\big)+\mathfrak l\big(\big(\beta^2(t)\cdot\alpha\beta(x)\big)\cdot\alpha^2\beta(z)\big)\alpha_V^3\beta_V(v)\\&&+(-1)^{|y|(|x|+|z|+|t|)}\mathfrak r\big(\alpha^3\beta(y)\big)\Big(\mathfrak l(\beta^2(t)\cdot\alpha\beta(x))\alpha_V^2\beta_V(w)+(-1)^{|z|(|t|+|x|)}\mathfrak r(\alpha^2\beta(z))\big(\mathfrak l(\beta^2(t))\alpha_V\beta_V(u)+(-1)^{|x||t|}\mathfrak r(\alpha\beta(x))\beta_V^2(\theta)\big)\Big)\\&&-\mathfrak l(\alpha(\beta^2(t)\cdot\alpha\beta(x)))\big(\mathfrak l(\alpha^2\beta(z))\alpha_V^3(v)+(-1)^{|y||z|}\mathfrak r(\alpha^3(y))\alpha_V^2\beta_V(w)\big)\\&&-(-1)^{(|x|+|t|)(|y|+|z|)}\mathfrak r\big(\alpha^2\beta(z)\cdot\alpha^3(y)\big)\big(\alpha_V\mathfrak l(\beta^2(t))\alpha_V\beta_V(u)+(-1)^{|x||t|}\alpha_V\mathfrak r(\alpha\beta(x))\beta_V^2(\theta)\big).    
\end{eqnarray*}
The binary operation $\cdot_{\mathfrak J\oplus V}$ satisfies condition \eqref{BiH-Jor-identity} on $\mathfrak J\oplus V$ if and only if $$\displaystyle\circlearrowleft_{x,y,t}(-1)^{|t|(|x|+|z|)}ass_{\mathfrak J\oplus V}\big((\beta\oplus\beta_V)^2(x+u)\cdot_{\mathfrak J\oplus V}(\alpha\oplus\alpha_V)(\beta\oplus\beta_V)(y+v),(\alpha\oplus\alpha_V)^2(\beta\oplus\beta_V)(z+w),(\alpha\oplus\alpha_V)^3(t+\theta)\big)=0,$$
Which gives by identification the following two independent conditions
\begin{eqnarray*}
&&\mathfrak l\big((\beta^2(x)\cdot\alpha\beta(y))\cdot\alpha^2\beta(z)\big)\alpha_V^3\beta_V(\theta)-\mathfrak l\big(\alpha\beta^2(x)\cdot\alpha^2\beta(y)\big)\mathfrak l(\alpha^2\beta(z))\alpha_V^3(\theta)\\
&&\quad=(-1)^{|z|(|x|+|y|)}\mathfrak r(\alpha^2\beta(z)\cdot\alpha^3(x))\alpha_V(\theta)\mathfrak l(\beta^2(y))\alpha_V\beta_V(\theta)-(-1)^{|y||z|}\mathfrak r(\alpha^3\beta(x))\mathfrak r(\alpha^2\beta(z))\mathfrak l(\beta^2(y))\alpha_V\beta_V(\theta)\\
&&\quad+(-1)^{|z|(|x|+|y|)}\mathfrak r(\alpha^2\beta(z)\cdot\alpha^3(y))\alpha_V(\theta)\mathfrak r(\alpha\beta(x))\beta_V^2-(-1)^{|x|(|y|+|z|)}\mathfrak r(\alpha^3\beta(y))\mathfrak r(\alpha^2\beta(z))\mathfrak r(\alpha\beta(x))\beta_V^2(\theta),
\end{eqnarray*}
and
\begin{eqnarray*}
&&\mathfrak r(\alpha^3\beta(t))\mathfrak l(\beta^2(x)\cdot\alpha\beta(y))\alpha_V^2\beta_V(w)-\mathfrak l(\alpha\beta^2(x)\cdot\alpha^2\beta(y))\mathfrak r(\alpha^3(t))\alpha_V^2\beta_V(w)\\
&&\quad=(-1)^{|z|(|x|+|y|)}\mathfrak l\big(\alpha(\beta^2(y)\cdot\alpha\beta(t))\big)\mathfrak r(\alpha^3(x))\alpha_V^2\beta_V(w)-(-1)^{|x|(|y|+|z|)}\mathfrak r(\alpha^3\beta(x))\mathfrak l(\beta^2(y)\cdot\alpha\beta(t))\alpha_V^2\beta_V(w)\\
&&\quad+\mathfrak l\big(\alpha(\beta^2(t)\cdot\alpha\beta(x))\big)r(\alpha^3(y))\alpha_V^2\beta_V(w)-(-1)^{|z|(|x|+|y|)}\mathfrak r(\alpha^3\beta(y))\mathfrak l(\beta^2(t)\cdot\alpha\beta(x))\alpha_V^2\beta_V(w).    
\end{eqnarray*}
These two conditions are equivalent respectively to conditions \eqref{con-bimod-J-4} and \eqref{con-bimod-J-5}. Thus, we conclude that, $\cdot_{\mathfrak J\oplus V},\;\alpha\oplus\alpha_V$ and $\beta\oplus\beta_V$ defines on $\mathfrak J\oplus V$ a BiHom-Jordan superalgebra structure if and only if $(V,\mathfrak l,\mathfrak r,\alpha_V,\beta_V)$ is a BiHom-super-bimodule of $(\mathfrak J,\cdot,\alpha,\beta)$.
\end{proof}
\begin{cor}
Let $(V,\mathfrak l,\mathfrak r,\alpha_V,\beta_V)$ be a BiHom-super-bimodule of a multiplicative BiHom-Jordan superalgebra $(\mathfrak J,\cdot,\alpha,\beta)$. Then, the linear maps $\mathfrak l$ and $\mathfrak r$ satisfying the following conditions   \begin{eqnarray}
&&\alpha_V\mathfrak l(x)=\mathfrak l(\alpha(x))\alpha_V~~~~\text{and}~~~~
\beta_V\mathfrak l(x)=\mathfrak l(\beta(x))\beta_V,\label{con-bimod-J-2}\\
&&\alpha_V\mathfrak r(x)=\mathfrak r(\alpha(x))\alpha_V~~~~\text{and}~~~~
\beta_V\mathfrak r(x)=\mathfrak r(\beta(x))\beta_V,\label{con-bimod-J-3}   
\end{eqnarray} 
for any $x,y\in\mathcal{H}(\mathfrak J)$.
\end{cor}

 Let $(V,\mathfrak l,\mathfrak r,\alpha_V,\beta_V)$ be a BiHom-super-bimodule of a multiplicative BiHom-Jordan superalgebra $(\mathfrak J,\cdot,\alpha,\beta)$. Then, by Theorem \ref{pr-semi-direct-BiH-Jord-Bim}, the product $\cdot_{\mathfrak J\oplus V}$ defined by Eq. \eqref{str-pr-semi-dir3} is BiHom-super-symmetric, which gives by a direct computation that
\begin{equation}\label{bim-BiH-Jor-equiv-repr}
\mathfrak l(x)v=\mathfrak r(\alpha\beta^{-1}(x))\alpha_V^{-1}\beta_V(v),\;\forall x\in\mathcal{H}(\mathfrak J),\;v\in\mathcal{H}(V),    
\end{equation}
which allows us to say the BiHom-Jordan superalgebra representation as we are interested to study in the following.
\begin{defn}
A representation of a regular multiplicative BiHom-Jordan superalgebra $(\mathfrak J,\cdot,\alpha,\beta)$ is a quadruple $(V,\rho,\alpha_V,\beta_V)$ consisting of a $\mathbb Z_2$-graded vector space $V$, an even linear map $\rho:\mathfrak J\to \mathfrak{gl}(V)$ and two even commuting automorphisms $\alpha_V,\beta_V\in \mathcal{A}ut(V)$ such that:
\begin{equation}\label{con-rep-J-1}
\alpha_V\rho(x)=\rho(\alpha(x))\alpha_V~~~~\text{and}~~~~
\beta_V\rho(x)=\rho(\beta(x))\beta_V,
\end{equation}
\begin{align}
\sum_{x,y,z}(-1)^{|x||z|}\rho\big(\alpha\beta^2(x)\cdot\alpha^2\beta(y)\big)\rho(\alpha^2\beta(z))\alpha_V^3=&(-1)^{|x||z|}\mathfrak \rho\big((\beta^2(x)\cdot\alpha\beta(y))\cdot\alpha^2\beta(z)\big)\alpha_V^3\beta_V\label{con-rep-J-2}\\&+(-1)^{|x|(|y|+|z|)}\rho(\alpha^2\beta^2(x))\rho(\alpha^2\beta(z))\rho(\alpha^2(y))\alpha_V^3\beta_V^{-1}\nonumber\\
&+(-1)^{|x||y|}\rho(\alpha^2\beta^2(y))\rho(\alpha^2\beta(z))\rho(\alpha^2(x))\alpha_V^3\beta_V^{-1},\nonumber\\
\sum_{x,y,z}(-1)^{|x||z|}\rho(\alpha^2\beta^2(x))\rho\big(\alpha\beta(y)\cdot\alpha^2(z)\big)\alpha_V^3
=&\sum_{x,y,z}(-1)^{|x||z|}\rho\big(\alpha\beta^2(x)\cdot\alpha^2\beta(y)\big)\rho(\alpha^2\beta(z))\alpha_V^3\label{con-rep-J-3},
\end{align}
for any $x,y,z\in \mathcal{H}(\mathfrak{J})$, where  $\displaystyle\sum_{x,y,z}$ denotes the cyclic sum over $(x,y,z)$.\\
Two representations $(V_1,\rho_1,\alpha_{V_1},\beta_{V_1})$ and $(V_2,\rho_2,\alpha_{V_2},\beta_{V_2})$ of a BiHom-Jordan superalgebra $(\mathfrak J,\cdot,\alpha,\beta)$ are called isomorphic
if there exists an even linear isomorphism $\Phi:V_1\to V_2$ satisfying
\begin{align}
\Phi\alpha_{V_1}=\alpha_{V_2}\Phi\;\;&\text{and}\;\;\Phi\beta_{V_1}=\beta_{V_1}\Phi,\label{isom-reps-1}\\
\Phi\rho_1(x)&=\rho_2(x)\Phi,\;\forall x\in\mathfrak J.\label{isom-reps-2}
\end{align}

\end{defn}
\begin{re}
If $\alpha=\beta$ and $\alpha_V=\beta_V$, we recover the representation of Hom-Jordan superalgebra (\cite{M-N-S}).   
\end{re}
\begin{ex}
 Let $(\mathfrak{J}, \cdot, \alpha,\beta) $ be a BiHom-Jordan superalgebra. Then, $(\mathfrak{J},ad,\alpha,\beta)$ where, the even linear map $ad:\mathfrak{J}\longrightarrow End(\mathfrak{J})$ is defined by
 $$ad(x)(y)= x\cdot y,~~\forall x,y\in \mathfrak{J}, $$
 is a representation of $\mathfrak{J}$  called the adjoint representation.
 \end{ex}
 Let $(V_1,\rho_1,\alpha_{V_1},\beta_{V_1})$  and $(V_2,\rho_2,\alpha_{V_2},\beta_V^2)$ be two representations of a BiHom-Jordan superalgebra $(\mathfrak J,\cdot,\alpha,\beta)$. Define the linear map $\rho_1\oplus\rho_2:\mathfrak J\to End(V_1\oplus V_2) $ by 
 \begin{equation}\label{direct-sum-rep}
 \rho_1\oplus\rho_2(x)(u,v):=(\rho_1(x)(u),\rho_2(x)v),\;\forall x\in\
  \mathfrak J,\;(u,v)\in V_1\times V_2.
 \end{equation}
\begin{prop}
Let $(V_1,\rho_1,\alpha_{V_1},\beta_{V_1})$  and $(V_2,\rho_2,\alpha_{V_2},\beta_{V_2})$ be two representations of a BiHom-Jordan superalgebra $(\mathfrak J,\cdot,\alpha,\beta)$. Then, $(V_1\oplus V_2,\rho_1\oplus\rho_2,\alpha_{V_1}\oplus\alpha_{V_2},\beta_{V_1}\oplus\beta_{V_2})$ is a representation of $\mathfrak J$ called direct sum representation of $(V_1,\rho_1,\alpha_{V_1},\beta_{V_1})$  and $(V_2,\rho_2,\alpha_{V_2},\beta_{V_2})$, where $\alpha_{V_1}\oplus\alpha_{V_2}$ and $\beta_{V_1}\oplus\beta_{V_2}$ are defined by Eq. \eqref{str-pr-semi-dir1}-\eqref{str-pr-semi-dir2} and $\rho_1\oplus\rho_2$ is defined by Eq. \eqref{direct-sum-rep}.  
\end{prop}
\begin{proof}
Let $x\in\mathfrak J$ and $(u,v)\in V_1\oplus V_2$. Then we have 
\begin{align*}
(\alpha_{V_1}\oplus\alpha_{V_2})(\rho_1\oplus\rho_2)(x)(u,v)&=(\alpha_{V_1}\oplus\alpha_{V_2})(\rho_1(x)(u),\rho_2(x)v)\\&=\alpha_{V_1}\rho_1(x)(u)+ \alpha_{V_2}\rho_2(x)(u)\\&= \rho_1(\alpha(x))\alpha_{V_1}(u)+\rho_2(\alpha(x))\alpha_{V_2}(u)\\&=\rho_1\oplus\rho_2(\alpha(x))(\alpha_{V_1}(u),\alpha_{V_2}(v))\\&= (\rho_1\oplus\rho_2)(\alpha(x))(\alpha_{V_1}\oplus\alpha_{V_2})(u,v).
\end{align*}
By the same way, we can show that $(\beta_{V_1}\oplus\beta_{V_2})(\rho_1\oplus\rho_2)(x)(u,v)=(\rho_1\oplus\rho_2)(\beta(x))(\beta_{V_1}\oplus\beta_{V_2})(u,v)$.\\

For any $x,y,z\in\mathfrak J$ and $(u,v)\in V_1\times V_2$, then by using Eq. \eqref{con-rep-J-2}, we have
\begin{align*}
&\sum_{x,y,z}(-1)^{|x||z|}(\rho_1\oplus\rho_2)\big(\alpha\beta^2(x)\cdot\alpha^2\beta(y)\big)(\rho_1\oplus\rho_2)(\alpha^2\beta(z))(\alpha_{V_1}\oplus\alpha_{V_2})^3(u,v)\\=& \sum_{x,y,z}(-1)^{|x||z|}(\rho_1\oplus\rho_2)\big(\alpha\beta^2(x)\cdot\alpha^2\beta(y)\big)\big(\rho_1(\alpha^2\beta(z))(\alpha_{V_1}^3(u)),\rho_2(\alpha^2\beta(z))(\alpha_{V_2}^3(v))\big)\\&=\Big(\sum_{x,y,z}(-1)^{|x||z|}\rho_1\big(\alpha\beta^2(x)\cdot\alpha^2\beta(y)\big)\rho_1\big(\alpha^2\beta(z)\big)\alpha_{V_1}^3(u),\sum_{x,y,z}(-1)^{|x||z|}\rho_2\big(\alpha\beta^2(x)\cdot\alpha^2\beta(y)\big)\rho_2\big(\alpha^2\beta(z)\big)\alpha_{V_2}^3(v)\Big)\\&=\Big((-1)^{|x||z|}\mathfrak \rho_1\big((\beta^2(x)\cdot\alpha\beta(y))\cdot\alpha^2\beta(z)\big)\alpha_{V_1}^3\beta_{V_1}(u)+(-1)^{|x|(|y|+|z|)}\rho_1(\alpha^2\beta^2(x))\rho_1(\alpha^2\beta(z))\rho_1(\alpha^2(y))\alpha_{V_1}^3\beta_{V_1}^{-1}(u)\\
&+(-1)^{|x||y|}\rho_1(\alpha^2\beta^2(y))\rho(\alpha^2\beta(z))\rho_1(\alpha^2(x))\alpha_{V_1}^3\beta_{V_1}^{-1}(u)\;,\;(-1)^{|x||z|}\mathfrak \rho_2\big((\beta^2(x)\cdot\alpha\beta(y))\cdot\alpha^2\beta(z)\big)\alpha_{V_2}^3\beta_{V_2}(v)\\
&+(-1)^{|x|(|y|+|z|)}\rho_2(\alpha^2\beta^2(x))\rho_2(\alpha^2\beta(z))\rho_2(\alpha^2(y))\alpha_{V_2}^3\beta_{V_2}^{-1}(v)\\&+(-1)^{|x||y|}\rho_2(\alpha^2\beta^2(y))\rho_2(\alpha^2\beta(z))\rho_2(\alpha^2(x))\alpha_{V_2}^3\beta_{V_2}^{-1}(v)\Big)\\&=(-1)^{|x||z|}\mathfrak (\rho_1\oplus\rho_2)\big((\beta^2(x)\cdot\alpha\beta(y))\cdot\alpha^2\beta(z)\big)(\alpha_{V_1}\oplus\alpha_{V_2})^3(\beta_{V_1}\oplus\beta_{V_2})(u,v)\\&+(-1)^{|x|(|y|+|z|)}(\rho_1\oplus\rho_2)(\alpha^2\beta^2(x))(\rho_1\oplus\rho_2)(\alpha^2\beta(z))(\rho_1\oplus\rho_2)(\alpha^2(y))(\alpha_{V_1}\oplus\alpha_{V_2})^3(\beta_{V_1}\oplus\beta_{V_2})^{-1}(u,v)\\
&+(-1)^{|x||y|}(\rho_1\oplus\rho_2)(\alpha^2\beta^2(y))(\rho_1\oplus\rho_2)(\alpha^2\beta(z))(\rho_1\oplus\rho_2)(\alpha^2(x))(\alpha_{V_1}\oplus\alpha_{V_2})^3(\beta_{V_1}\oplus\beta_{V_2})^{-1}(u,v),    
\end{align*}
which gives that, the condition \eqref{con-rep-J-2} is satisfied by $\rho_1\oplus\rho_2$. Similarly, we can show that the condition \eqref{con-rep-J-3} is satified. The Proposition is proved.
\end{proof}

\section{Dual representation of BiHom-Jordan superalgebras}\label{Sec3}
In this section, we interest to construct the dual representation of a representation of multiplicative regular BiHom-Jordan superalgebra without any additional condition and we give some related fundamental results.\\

Let $(V,\rho,\alpha_V,\beta_V)$ be a representation of a multiplicative regular Bihom-Jordan superalgebra $(\mathfrak{J},\cdot,\alpha,\beta)$. Let us consider $V^*$ the dual space of $V$ and $\tilde{\alpha}_V,\tilde{\beta}_V:V^{*}\rightarrow V^{*}$ two homomorphisms defined by
 $\tilde{\alpha}_V(f)=f\circ \alpha_V$,  $\tilde{\beta}_V(f)=f\circ \beta_V$, $\forall f\in V^{*}$. Define the even linear map $\widetilde{\rho}:\mathfrak{J}\rightarrow \mathfrak{gl}(V^{\ast})$ by $$\widetilde{\rho}(x)(f)=(-1)^{|f||x|}f\circ\rho(x),~\forall x,y\in\mathcal{H}(\mathfrak{J}), f\in V^{\ast}.$$
 
 However, in general $\widetilde{\rho}$ is not a representation of $\mathfrak{J}$. Define
\begin{equation}\label{def-Jor-dual-repr}
\rho^\star:\mathfrak{J}\longrightarrow End(V^*):\,\rho^\star(x)(\xi):=\widetilde{\rho}(\beta(x))(\beta_{V}^{-2})^*(\xi),\;\forall x\in \mathcal{H}(\mathfrak{J}),\xi\in \mathcal{H}(V^*).
\end{equation}
More precisely 
$$\langle\rho^\star(x)(\xi),u\rangle=\langle \xi,\rho(\beta^{-3}(x))(\beta_V^{-2}(u))\rangle,\;\forall x\in \mathcal{H}(\mathfrak{J}),\xi\in \mathcal{H}(V^*),u\in \mathcal{H}(V).$$
\begin{thm}
Let $(V,\rho,\alpha_V,\beta_V)$ be a representation of a multiplicative regular BiHom-Jordan superalgebra $(\mathfrak{J},\cdot,\alpha,\beta)$. Then $(V^*,\rho^\star,(\alpha_V^{-1})^*,(\beta_V^{-1})^*)$ is a representation of $(\mathfrak{J},\cdot,\alpha,\beta)$ on $V^*$, where $\rho^\star$ is defined above by \eqref{def-Jor-dual-repr}.    
\end{thm}
\begin{proof}
 For any $x\in\mathcal H(\mathfrak J)$ and $\xi\in\mathcal{H}(V^*)$, we have 
 $$
\rho^\star(\alpha(x))(\alpha_V^{-1})^*(\xi)=\widetilde{\rho}(\alpha\beta(x))(\alpha_V^{-1})^*(\beta_V^{-2})^*(\xi)=(\alpha_V^{-1})^*\widetilde{\rho}(\beta(x))(\beta_V^{-2})^*(\xi)=(\alpha_V^{-1})^*\rho^\star(x)(\xi),$$
and 
$$\rho^\star(\beta(x))(\beta_V^{-1})^*(\xi)=\widetilde{\rho}(\beta^2(x))(\beta_V^{-3})^*(\xi)=(\beta_V^{-1})^*\widetilde{\rho}(\beta(x))(\beta_V^{-2})^*(\xi)=(\beta_V^{-1})^*\rho^\star(x)(\xi).$$
Let $x,y,z\in\mathcal{H}(\mathfrak J),\;u\in\mathcal{H}(V)$ and $\xi\in\mathcal{H}(V^*)$, then, by using Eq. \eqref{con-rep-J-2}, \eqref{con-rep-J-3} and \eqref{def-Jor-dual-repr}, we have
\begin{align*}
\langle\sum_{x,y,z}&(-1)^{|x||z|}\rho^\star\big(\alpha\beta^2(x)\cdot\alpha^2\beta(y)\big)\rho^\star(\alpha^2\beta(z))(\alpha_V^{-3})^*(\xi),u\rangle\\=&\sum_{x,y,z}(-1)^{|x||z|}\langle\widetilde{\rho}\big(\alpha\beta^3(x)\cdot\alpha^2\beta^2(y)\big)(\beta_V^{-2})^*\widetilde{\rho}(\alpha^2\beta^2(z))(\alpha_V^{-3}\beta^{-2})^*(\xi),u\rangle\\=&\sum_{x,y,z}(-1)^{|x||z|}\langle\widetilde{\rho}\big(\alpha\beta^3(x)\cdot\alpha^2\beta^2(y)\big)\widetilde{\rho}(\alpha^2\beta^4(z))(\alpha_V^{-3}\beta^{-4})^*(\xi),u\rangle\\=& \sum_{x,y,z}(-1)^{|y||z|}(-1)^{|\xi|(|x|+|y|+|z|)}\langle\xi,(\alpha_V^{-3}\beta^{-4})\rho(\alpha^2\beta^4(z))\rho\big(\alpha\beta^3(x)\cdot\alpha^2\beta^2(y)\big)(u)\rangle\\=& \sum_{x,y,z}(-1)^{|y||z|}(-1)^{|\xi|(|x|+|y|+|z|)}\langle\xi,\rho(\alpha^{-1}(z))\rho\big(\alpha^{-2}\beta^{-1}(x)\cdot\alpha^{-1}\beta^{-2}(y)\big)(\alpha_V^{-3}\beta^{-4})(u)\rangle. 
\end{align*}
We take $X=\alpha^{-3}\beta^{-2}(x),\;Y=\alpha^{-3}\beta^{-2}(y)$ and $Z=\alpha^{-3}\beta^{-2}(z)$, then we have
\begin{align*}
\langle\sum_{x,y,z}&(-1)^{|x||z|}\rho^\star\big(\alpha\beta^2(x)\cdot\alpha^2\beta(y)\big)\rho^\star(\alpha^2\beta(z))(\alpha_V^{-3})^*(\xi),u\rangle\\=& \sum_{x,y,z}(-1)^{|y||z|}(-1)^{|\xi|(|x|+|y|+|z|)}\langle\xi,\rho(\alpha^2\beta^2(Z))\rho\big(\alpha\beta(X)\cdot\alpha^2(Y)\big)(\alpha_V^{-3}\beta_V^{-4})(u)\rangle\\=& \sum_{x,y,z}(-1)^{|x||z|}(-1)^{|\xi|(|x|+|y|+|z|)}\langle\xi,\rho(\alpha^2\beta^2(X))\rho\big(\alpha\beta(Y)\cdot\alpha^2(Z)\big)(\alpha_V^{-3}\beta_V^{-4})(u)\rangle\\\stackrel{\eqref{con-rep-J-3}}{=}& \sum_{x,y,z}(-1)^{|x||z|}(-1)^{|\xi|(|x|+|y|+|z|)}\langle\xi,\rho\big(\alpha\beta^2(X)\cdot\alpha^2\beta(Y)\big)\rho(\alpha^2\beta(Z))(\alpha_V^{-3}\beta_V^{-4})(u)\rangle\\\stackrel{\eqref{con-rep-J-2}}{=}&(-1)^{|x||z|}(-1)^{|\xi|(|x|+|y|+|z|)}\langle\xi,\rho\big((\beta^2(X)\cdot\alpha\beta(Y))\cdot\alpha^2\beta(Z)\big)(\alpha_V^{-3}\beta^{-3})(u)\rangle\\&+(-1)^{|x|(|y|+|z|)}(-1)^{|\xi|(|x|+|y|+|z|)}\langle\xi,\rho(\alpha^2\beta^2(X))\rho(\alpha^2\beta(Z))\rho(\alpha^2(Y))(\alpha_V^{-3}\beta_V^{-5})(u)\rangle\\&+(-1)^{|x||y|}(-1)^{|\xi|(|x|+|y|+|z|)}\langle\xi,\rho(\alpha^2\beta^2(Y))\rho(\alpha^2\beta(Z))\rho(\alpha^2(X))(\alpha_V^{-3}\beta_V^{-5})(u)\rangle \\=&(-1)^{|x||z|}\langle(\alpha_V^{-3}\beta^{-3})^*\widetilde{\rho}\big((\beta^2(X)\cdot\alpha\beta(Y))\cdot\alpha^2\beta(Z)\big)(\xi),u\rangle\\&+(-1)^{|y||z|}\langle(\alpha_V^{-3}\beta_V^{-5})^*\widetilde{\rho}(\alpha^2(Y))\widetilde{\rho}(\alpha^2\beta(Z))\widetilde{\rho}(\alpha^2\beta^2(X))(\xi),u\rangle\\&+(-1)^{|z|(|x|+|y|)}\langle(\alpha_V^{-3}\beta_V^{-5})^*\widetilde{\rho}(\alpha^2(X))\widetilde{\rho}(\alpha^2\beta(Z))\widetilde{\rho}(\alpha^2\beta^2(Y))(\xi),u\rangle\\=&(-1)^{|x||z|}\langle\widetilde{\rho}\big((\alpha^3\beta^5(X)\cdot\alpha^4\beta^4(Y))\cdot\alpha^5\beta^4(Z)\big)(\alpha_V^{-3}\beta^{-3})^*(\xi),u\rangle\\&+(-1)^{|y||z|}\langle\widetilde{\rho}(\alpha^5\beta^5(Y))(\beta_V^{-2})^*\widetilde{\rho}(\alpha^5\beta^4(Z))(\beta_V^{-2})^*\widetilde{\rho}(\alpha^5\beta^3(X))(\alpha_V^{-3}\beta_V^{-1})^*(\xi),u\rangle\\&+(-1)^{|z|(|x|+|y|)}\langle\widetilde{\rho}(\alpha^5\beta^5(x))(\beta_V^{-2})^*\widetilde{\rho}(\alpha^5\beta^4(Z))(\beta_V^{-2})^*\widetilde{\rho}(\alpha^5\beta^3(y))(\alpha_V^{-3}\beta_V^{-1})^*(\xi),u\rangle\\=& (-1)^{|x||z|}\langle\rho^\star\big((\alpha^3\beta^4(X)\cdot\alpha^4\beta^3(Y))\cdot\alpha^5\beta^3(Z)\big)(\alpha_V^{-3}\beta^{-1})^*(\xi),u\rangle\\&+(-1)^{|y||z|}\langle\rho^\star(\alpha^5\beta^4(Y))\rho^\star(\alpha^5\beta^3(Z))\widetilde{\rho}(\alpha^5\beta^2(X))(\alpha_V^{-3}\beta_V)^*(\xi),u\rangle\\&+(-1)^{|z|(|x|+|y|)}\langle\rho^\star(\alpha^5\beta^4(X))\rho^\star(\alpha^5\beta^3(Z))\widetilde{\rho}(\alpha^5\beta^2(Y))(\alpha_V^{-3}\beta_V)^*(\xi),u\rangle\\=& (-1)^{|x||z|}\langle\rho^\star\big((\beta^2(x)\cdot\alpha\beta(y))\cdot\alpha^2\beta(z)\big)(\alpha_V^{-3}\beta^{-1})^*(\xi),u\rangle\\&+(-1)^{|y||z|}\langle\rho^\star(\alpha^2\beta^2(y))\rho^\star(\alpha^2\beta(z))\widetilde{\rho}(\alpha^2(x))(\alpha_V^{-3}\beta_V)^*(\xi),u\rangle\\&+(-1)^{|z|(|x|+|y|)}\langle\rho^\star(\alpha^2\beta^2(x))\rho^\star(\alpha^2\beta(z))\widetilde{\rho}(\alpha^2(y))(\alpha_V^{-3}\beta_V)^*(\xi),u\rangle.  
\end{align*}
Then $\rho^\star$ satisfying \eqref{con-rep-J-2}. Similarly, we can show that, $\rho^\star$ satisfying \eqref{con-rep-J-3}, which gives the proof.
\end{proof}
\begin{lem}\label{dual-dual-Jor-repres}
Let $(V,\rho,\alpha_v,\beta_V)$ be a representation of a multiplicative regular Bihom-Jordan superalgebra $(\mathfrak{J},\cdot,\alpha,\beta)$. Then we have $$(\rho^\star)^\star=\rho.$$ 
\end{lem}
\begin{proof}
Let $x\in\mathcal{H}(\mathfrak{J}),\;\xi\in V^*$ and $u\in\mathcal{H}(V)$. Then, we have 
\begin{align*}
<(\rho^\star)^\star(x)(u),\xi>&=<\widetilde{\rho^\star}(\beta(x))(\beta_V^2(u)),\xi>\\
&=(-1)^{|u||x|}<\beta_V^2(u),\rho^\star(\beta(x))(\xi)>\\&=(-1)^{|u||x|}<\beta_V^2(u),\widetilde{\rho}(\beta^2(x))(\beta_V^{-2})^*(\xi)>\\&=<\beta_V^{-2}\rho(\beta^2(x))(\beta_V^2)(u),\xi>\\&=<\rho(x)(u),\xi>,    
\end{align*}
which implies that $(\rho^\star)^\star=\rho$.
\end{proof}
\begin{cor}
Let $(\mathfrak J,\cdot,\alpha,\beta)$ be a BiHom-Jordan superalgebra. Then $ad^\star:\mathfrak J\to\mathfrak{gl}(\mathfrak J^\ast)$ defined by
\begin{equation}\label{dual-adj-rep}
 ad^\star(x)(\xi):= \widetilde{ad}(\beta(x))(\beta^{-2})^\ast(\xi),\;\forall x\in\mathcal{H}(\mathfrak J),\;\xi\in\mathcal{H}(\mathfrak J^\ast),
\end{equation}
is a representation of the BiHom-Jordan superalgebra $(\mathfrak J,\cdot,\alpha,\beta)$ on $\mathfrak J^\ast$ with respect to $((\alpha^{-1})^\ast,(\beta^{-1})^\ast)$ called coadjoint representation.
\end{cor}

Using the coadjoint representation $ad^\star$, we can obtain a semidirect product BiHom-Jordan superalgebra structure on $\mathfrak J\oplus\mathfrak J^\ast$.
\begin{cor}
Let $(\mathfrak J,\cdot,\alpha,\beta)$ be a BiHom-Jordan superalgebra. Then, the quadruple $(\mathfrak J\oplus\mathfrak J^\ast,\bullet_{\mathfrak J\oplus\mathfrak J^\ast},\alpha\oplus(\alpha^{-1})^\ast,\beta\oplus(\beta^{-1})^\ast)$ is a BiHom-Jordan superalgebra, where the product $\bullet_{\mathfrak J\oplus\mathfrak J^\ast}$ is defined by 
\begin{equation}\label{struc-sem-dir-BiH-Jor-ad}
 \bullet_{\mathfrak J\oplus\mathfrak J^\ast}(x+\xi,y+\eta)=x\cdot y+\rho^\star(x)(\eta)+(-1)^{|\xi||y|}\rho^\star(y)(\xi),   
\end{equation}
for all $x,y\in\mathcal{H}(\mathfrak J),\;\xi,\eta\in\mathcal{H}(\mathfrak J^\ast)$.
\end{cor}
\section{BiHom-pre-Jordan superalgebras and homogeneous $\mathcal{O}$-operators}\label{Sec4}
\subsection{BiHom-pre-Jordan superalgebras }
\begin{defn}
A BiHom-pre-Jordan superalgebra is a quadruple $(\mathfrak J,\circ,\alpha,\beta)$ consisting of a $\mathbb Z_2$-graded vector space $\mathfrak J$, an even bilinear map $\circ:\mathfrak J\times\
\mathfrak J\to \mathfrak J$ and two even commuting linear maps $\alpha,\beta:\mathfrak J\to\mathfrak J$ such that the following conditions hold:
\begin{equation}\label{con-pre-Jord-1}
\alpha(x\circ y)=(\alpha(x)\circ \alpha(y))~~~~\text{and}~~~~
\beta(x\circ y)=(\beta(x)\circ \beta(y)),
\end{equation}
\begin{align}
\sum_{x,y,z}(-1)^{|x||z|}\big(\alpha\beta^2(x)\ast\alpha^2\beta(y)\big)\circ(\alpha^2\beta(z)\circ\alpha^3(t))=&(-1)^{|x||z|} \big((\beta^2(x)\ast\alpha\beta(y))\ast\alpha^2\beta(z)\big)\circ\alpha^3\beta(t)\label{con-BiH-pre-Jord-2}\\&+(-1)^{|x|(|y|+|z|)}(\alpha^2\beta^2(x))\circ\big(\alpha^2\beta(z)\circ(\alpha^2(y)\circ\alpha^3\beta^{-1}(t))\big)\nonumber\\
&+(-1)^{|x||y|}(\alpha^2\beta^2(y))\circ\big(\alpha^2\beta(z)\circ(\alpha^2(x)\circ\alpha^3\beta^{-1}(t))\big),\nonumber\\
\sum_{x,y,z}(-1)^{|x||z|}\alpha^2\beta^2(x)\circ\big((\alpha\beta(y)\ast\alpha^2(z))\circ\alpha^3(t)\big)
=&\sum_{x,y,z}(-1)^{|x||z|}\big(\alpha\beta^2(x)\ast\alpha^2\beta(y)\big)\circ(\alpha^2\beta(z)\circ\alpha^3(t))\label{con-BiH-pre-Jord-3},
\end{align}
for any $x,y,z,t\in \mathcal{H}(\mathfrak{J})$, where  $\displaystyle\sum_{x,y,z}$ denotes the cyclic sum over $(x,y,z)$ and the bilinear map $"\ast"$ is defined by 
\begin{equation}\label{BiH-pre-Jor-to-BiH-Jord}
x\ast y=x\circ y+(-1)^{|x||y|}\alpha^{-1}\beta(y)\circ\alpha\beta^{-1}(x).
\end{equation}
\end{defn}
\begin{thm}
 Let $(\mathfrak J,\circ,\alpha,\beta)$ be a BiHom-pre-Jordan superalgebra, then, the blinear map defined by Eq. \eqref{BiH-pre-Jor-to-BiH-Jord} defines on $\mathfrak J$ a BiHom-Jordan superalgebra structure which is called the associated BiHom-Jordan
superalgebra of $(\mathfrak J, \circ,\alpha,\beta)$ and $(\mathfrak J, \circ,\alpha,\beta)$ is called a compatible BiHom-pre-Jordan superalgebra structure of the BiHom-Jordan superalgebra $(\mathfrak J,\ast,\alpha,\beta)$.   
\end{thm}
\begin{proof}
The BiHom-super-symmetry of $\ast$ is obvious. Let $x,y,z,t\in\mathcal{H}(\mathfrak J)$, then
\begin{align*}
\displaystyle\circlearrowleft_{x,y,t}(-1)^{|t|(|x|+|z|)}ass_{\alpha,\beta}(\beta^2(x)\ast\alpha\beta(y),\alpha^2\beta(z),\alpha^3(t))=&(-1)^{|t|(|x|+|z|)}ass_{\alpha,\beta}(\beta^2(x)\ast\alpha\beta(y),\alpha^2\beta(z),\alpha^3(t))\\&+(-1)^{|x|(|y|+|z|)}ass_{\alpha,\beta}(\beta^2(y)\ast\alpha\beta(t),\alpha^2\beta(z),\alpha^3(x))\\&+(-1)^{|y|(|t|+|z|)}ass_{\alpha,\beta}(\beta^2(t)\ast\alpha\beta(x),\alpha^2\beta(z),\alpha^3(y)).
\end{align*}
By using Eq. \eqref{BiH-pre-Jor-to-BiH-Jord}, a direct computation gives
\begin{align*}
M=&ass_{\alpha,\beta}(\beta^2(x)\ast\alpha\beta(y),\alpha^2\beta(z),\alpha^3(t))\\=&\big((\beta^2(x)\ast\alpha\beta(y))\ast\alpha^2\beta(z)\big)\ast\alpha^3\beta(t)-(\alpha\beta^2(x)\ast\alpha^2\beta(y))\ast(\alpha^2\beta(z)\ast\alpha^3(t))\\=&\big((\beta^2(x)\ast\alpha\beta(y))\ast\alpha^2\beta(z)\big)\circ\alpha^3\beta(t)+(-1)^{|t|(|x|+|y|+|z|)}\alpha^2\beta^2(t)\circ\big((\alpha\beta(x)\ast\alpha^2(y))\ast\alpha^3(z)\big)\\&-(\alpha\beta^2(x)\ast\alpha^2\beta(y))\circ(\alpha^2\beta(z)\ast\alpha^3(t))-(-1)^{(|x|+|y|)(|z|+|t|)}(\alpha\beta^2(z)\ast\alpha^2\beta(t))\circ(\alpha^2\beta(x)\ast\alpha^3(y))\\=&\big((\beta^2(x)\ast\alpha\beta(y))\ast\alpha^2\beta(z)\big)\circ\alpha^3\beta(t)+(-1)^{|t|(|x|+|y|+|z|)}\alpha^2\beta^2(t)\circ\big((\alpha\beta(x)\ast\alpha^2(y))\circ\alpha^3(z))\big)\\&+(-1)^{(|x|+|y|)(|z|+|t|)+|z||t|}\alpha^2\beta^2(t)\circ\big(\alpha^2\beta(z)\circ(\alpha^2(x)\circ\alpha^3\beta^{-1}(y))\big)\\&+(-1)^{(|x|+|y|)(|z|+|t|)}(-1)^{|x||y|+|z||t|}\alpha^2\beta^2(t)\circ\big(\alpha^2\beta(z)\circ(\alpha^2(y)\circ\alpha^3\beta^{-1}(x))\big)\\&-(\alpha\beta^2(x)\ast\alpha^2\beta(y))\circ(\alpha^2\beta(z)\circ\alpha^3(t))-(-1)^{|z||t|}(\alpha\beta^2(x)\ast\alpha^2\beta(y))\circ(\alpha^2\beta(t)\circ\alpha^3(z))\\&-(-1)^{(|x|+|y|)(|z|+|t|)}(\alpha\beta^2(z)\ast\alpha^2\beta(t))\circ(\alpha^2\beta(x)\circ\alpha^3(y))\\&-(-1)^{(|x|+|y|)(|z|+|t|)+|x||y|}(\alpha\beta^2(z)\ast\alpha^2\beta(t))\circ(\alpha^2\beta(y)\circ\alpha^3(x)).
\end{align*}
Similarly, we have
\begin{align*}
N=&ass_{\alpha,\beta}(\beta^2(y)\ast\alpha\beta(t),\alpha^2\beta(z),\alpha^3(x))\\=&\big((\beta^2(y)\ast\alpha\beta(t))\ast\alpha^2\beta(z)\big)\circ\alpha^3\beta(x)+(-1)^{|x|(|y|+|t|+|z|)}\alpha^2\beta^2(x)\circ\big((\alpha\beta(y)\ast\alpha^2(t))\circ\alpha^3(z))\big)\\&+(-1)^{(|y|+|t|)(|z|+|x|)+|z||x|}\alpha^2\beta^2(x)\circ\big(\alpha^2\beta(z)\circ(\alpha^2(y)\circ\alpha^3\beta^{-1}(t))\big)\\&+(-1)^{(|y|+|t|)(|z|+|x|)}(-1)^{|y||t|+|z||x|}\alpha^2\beta^2(x)\circ\big(\alpha^2\beta(z)\circ(\alpha^2(t)\circ\alpha^3\beta^{-1}(y))\big)\\&-(\alpha\beta^2(y)\ast\alpha^2\beta(t))\circ(\alpha^2\beta(z)\circ\alpha^3(x))-(-1)^{|z||x|}(\alpha\beta^2(y)\ast\alpha^2\beta(t))\circ(\alpha^2\beta(x)\circ\alpha^3(z))\\&-(-1)^{(|y|+|t|)(|z|+|x|)}(\alpha\beta^2(z)\ast\alpha^2\beta(x))\circ(\alpha^2\beta(y)\circ\alpha^3(t))\\&-(-1)^{(|y|+|t|)(|z|+|x|)+|y||t|}(\alpha\beta^2(z)\ast\alpha^2\beta(x))\circ(\alpha^2\beta(t)\circ\alpha^3(y)), 
\end{align*}
and
\begin{align*}
P=&ass_{\alpha,\beta}(\beta^2(t)\ast\alpha\beta(x),\alpha^2\beta(z),\alpha^3(y))\\=&\big((\beta^2(t)\ast\alpha\beta(x))\ast\alpha^2\beta(z)\big)\circ\alpha^3\beta(y)+(-1)^{|y|(|x|+|t|+|z|)}\alpha^2\beta^2(y)\circ\big((\alpha\beta(t)\ast\alpha^2(x))\circ\alpha^3(z))\big)\\&+(-1)^{(|x|+|t|)(|z|+|y|)+|z||y|}\alpha^2\beta^2(y)\circ\big(\alpha^2\beta(z)\circ(\alpha^2(t)\circ\alpha^3\beta^{-1}(x))\big)\\&+(-1)^{(|x|+|t|)(|z|+|y|)}(-1)^{|x||t|+|z||y|}\alpha^2\beta^2(y)\circ\big(\alpha^2\beta(z)\circ(\alpha^2(x)\circ\alpha^3\beta^{-1}(t))\big)\\&-(\alpha\beta^2(t)\ast\alpha^2\beta(x))\circ(\alpha^2\beta(z)\circ\alpha^3(y))-(-1)^{|z||y|}(\alpha\beta^2(t)\ast\alpha^2\beta(x))\circ(\alpha^2\beta(y)\circ\alpha^3(z))\\&-(-1)^{(|x|+|t|)(|z|+|y|)}(\alpha\beta^2(z)\ast\alpha^2\beta(y))\circ(\alpha^2\beta(t)\circ\alpha^3(x))\\&-(-1)^{(|x|+|t|)(|z|+|y|)+|x||t|}(\alpha\beta^2(z)\ast\alpha^2\beta(y))\circ(\alpha^2\beta(x)\circ\alpha^3(t)).
\end{align*}
Then,
\begin{align*}
\displaystyle\circlearrowleft_{x,y,t}(-1)^{|t|(|x|+|z|)}ass_{\alpha,\beta}(\beta^2(x)\ast\alpha\beta(y),\alpha^2\beta(z),\alpha^3(t))=&(M_1-M_2)+(N_1-N_2)+(P_1-P_2)+(Q_1-Q_2),
\end{align*}
where,
\begin{align*}
 M_1=& (-1)^{|t|(|x|+|z|)}\big((\beta^2(x)\ast\alpha\beta(y))\ast\alpha^2\beta(z)\big)\circ\alpha^3\beta(t)+(-1)^{|z|(|y|+|t|)+|x||t|}\alpha^2\beta^2(x)\circ\big(\alpha^2\beta(z)\circ(\alpha^2(y)\circ\alpha^3\beta^{-1}(t))\big)\\&+(-1)^{|x|(|y|+|z|+|t|)+|z||t|}\alpha^2\beta^2(y)\circ\big(\alpha^2\beta(z)\circ(\alpha^2(x)\circ\alpha^3\beta^{-1}(t))\big),  
\end{align*}
\begin{align*}
 M_2=&(-1)^{|t|(|x|+|z|)}(\alpha\beta^2(x)\ast\alpha^2\beta(y))\circ(\alpha^2\beta(z)\circ\alpha^3(t))+(-1)^{|z|(|x|+|y|+|t|)+|x||t|}(\alpha\beta^2(z)\ast\alpha^2\beta(x))\circ(\alpha^2\beta(y)\circ\alpha^3(t))\\&+(-1)^{|x|(|y|+|z|+|t|)}(-1)^{|z|(|y|+|t|)}(\alpha\beta^2(z)\ast\alpha^2\beta(y))\circ(\alpha^2\beta(x)\circ\alpha^3(t)),   
\end{align*}
\begin{align*}
N_1=&(-1)^{|y|(|t|+|z|)}\big((\beta^2(t)\ast\alpha\beta(x))\ast\alpha^2\beta(z)\big)\circ\alpha^3\beta(y)+ (-1)^{|z|(|x|+|y|)+|t||y|}\alpha^2\beta^2(t)\circ\big(\alpha^2\beta(z)\circ(\alpha^2(x)\circ\alpha^3\beta^{-1}(y))\\&+ (-1)^{|t|(|x|+|y|+|z|)+|y||z|}\alpha^2\beta^2(x)\circ\big(\alpha^2\beta(z)\circ(\alpha^2(t)\circ\alpha^3\beta^{-1}(y))\big),  
\end{align*}
\begin{align*}
N_2=&(-1)^{|y|(|t|+|z|)}(\alpha\beta^2(t)\ast\alpha^2\beta(x))\circ(\alpha^2\beta(z)\circ\alpha^3(y))+(-1)^{|t|(|y|+|z|)+|y||z|}(\alpha\beta^2(z)\ast\alpha^2\beta(t))\circ(\alpha^2\beta(x)\circ\alpha^3(y))\\&+ (-1)^{|t|(|x|+|y|+|z|)}(-1)^{|z|(|x|+|y|)}(\alpha\beta^2(z)\ast\alpha^2\beta(x))\circ(\alpha^2\beta(t)\circ\alpha^3(y)),\end{align*}
\begin{align*}
 P_1=& (-1)^{|x|(|y|+|z|)}\big((\beta^2(y)\ast\alpha\beta(t))\ast\alpha^2\beta(z)\big)\circ\alpha^3\beta(x)+ (-1)^{|y|(|x|+|z|+|t|)+|x||z|}\alpha^2\beta^2(t)\circ\big(\alpha^2\beta(z)\circ(\alpha^2(y)\circ\alpha^3\beta^{-1}(x))\big)\\&+ (-1)^{|z|(|x|+|t|)+|x||y|}\alpha^2\beta^2(y)\circ\big(\alpha^2\beta(z)\circ(\alpha^2(t)\circ\alpha^3\beta^{-1}(x))\big),
\end{align*}
\begin{align*}
 P_2=& (-1)^{|x|(|y|+|z|)}(\alpha\beta^2(y)\ast\alpha^2\beta(t))\circ(\alpha^2\beta(z)\circ\alpha^3(x))\\&+(-1)^{|z|(|x|+|y|+|t|)}(-1)^{|y|(|x|+|t|)}(\alpha\beta^2(z)\ast\alpha^2\beta(t))\circ(\alpha^2\beta(y)\circ\alpha^3(x))\\&+  (-1)^{|z|(|x|+|y|+|t|)+|x||y|}(\alpha\beta^2(z)\ast\alpha^2\beta(y))\circ(\alpha^2\beta(t)\circ\alpha^3(x)),
\end{align*}
\begin{align*}
Q_1=&(-1)^{|y||t|}\alpha^2\beta^2(t)\circ\big((\alpha\beta(x)\ast\alpha^2(y))\circ\alpha^3(z))\big)+(-1)^{|x||t|}\alpha^2\beta^2(x)\circ\big((\alpha\beta(y)\ast\alpha^2(t))\circ\alpha^3(z))\big)\\&+(-1)^{|x||y|}\alpha^2\beta^2(y)\circ\big((\alpha\beta(t)\ast\alpha^2(x))\circ\alpha^3(z))\big),
\end{align*}
and
\begin{align*}
 Q_2=& (-1)^{|x||t|}(\alpha\beta^2(x)\ast\alpha^2\beta(y))\circ(\alpha^2\beta(t)\circ\alpha^3(z))+ (-1)^{|x||y|}(\alpha\beta^2(y)\ast\alpha^2\beta(t))\circ(\alpha^2\beta(x)\circ\alpha^3(z))\\+&(-1)^{|y||t|}(\alpha\beta^2(t)\ast\alpha^2\beta(x))\circ(\alpha^2\beta(y)\circ\alpha^3(z)). 
\end{align*}
By using Eq. \eqref{con-BiH-pre-Jord-2}, we get $M_1-M_2=N_1-N_2=P_1-P_2=0$ and Eq. \eqref{con-BiH-pre-Jord-3} gives that $Q_1-Q_2=0$. The Theorem is proved.
\end{proof}

\subsection{Homogeneous $\mathcal{O}$-operators of BiHom-pre-Jordan superalgebras }

In this subsection, we are ineresting to introduce the notion of homogeneous $\mathcal{O}$-operators of a multiplicative regular BiHom-Jordan superalgebras and we construct BiHom-pre-Jordan superalgebras from homogeneous $\mathcal{O}$-operators of BiHom-Jordan superalgebras.\\

\textbf{Parity reverse representations and parity dualities (\cite{Bai-Guo-Zhang}):} Let $V=V_{\overline{0}}\oplus V_{\overline{1}}$ be a $\mathbb{Z}_2$-graded vector space, we denote by $sV$ the $\mathbb{Z}_2$-graded vector space obtained by interchanging the even and odd parts of $V$,
that is, $(sV)_{\overline{0}}= V_{\overline{1}}$ and $(sV)_{\overline{1}}= V_{\overline{0}}$. The parity reverse map (or the suspension operator) is an
odd linear map $s:V\to sV$ that sends each homogeneous element of $V$ to the same element in
$sV$ but with the opposite parity, that is, $s$ sends $v\in V_j$ to an element $sv\in (sV)_{j+1}$ for all $j\in\mathbb{Z}$. In \cite{Bai-Guo-Zhang}, the authors gave the notion of the parity reverse of a representation of a Lie superalgebra,
obtained from the original representation by a suspension process. They also obtained
a natural one-one correspondence between $\mathcal{O}$-operators with any given parity associated to a
representation and $\mathcal{O}$-operators with the opposite parity associated to the parity reverse representation. In this subsection, we study these results on BiHom-Jordan superalgebras and then present some other related results.
\begin{defn}
Let $(\mathfrak J,\cdot,\alpha,\beta)$ be a BiHom-Jordan superalgebra and $(V,\rho,\alpha_V,\beta_V)$ be a representation of $\mathfrak J$. A homogeneous linear map $T:V\to\mathfrak J$ is called $\mathcal{O}$-operator of $\mathfrak J$ associated to $\rho$ if it satisfies
\begin{align}
& \alpha T=T\alpha_V\;\;\text{and}\;\;\;\beta T=T\beta_V,\label{cond-O-oper-BiH-Jor1}\\
& T(u)\cdot T(v)=T\big((-1)^{|T|(|T|+|u|)}\rho(T(u))v+(-1)^{|u|(|v|+|T|)}\rho(T(\alpha_V^{-1}\beta_V(v)))\alpha_V\beta_V^{-1}(u)\big),\;\forall u,v\in\mathcal{H}(V).\label{cond-O-oper-BiH-Jor2}
\end{align}

\end{defn}
An $\mathcal{O}$-operator is called even (res. odd) if it is an even (rep. odd) linear map. Let us denote by $\mathcal{O}_{\overline{0}}(\mathfrak J; V,\rho,\alpha_V,\beta_V)$ and $\mathcal{O}_{\overline{1}}(\mathfrak J; V,\rho,\alpha_V,\beta_V)$ respectively the sets of even and odd $\mathcal{O}$-operators on $\mathfrak J$ with respect to the representation $(V,\rho,\alpha_V,\beta_V)$.
\begin{re}
 A Rota-Baxter operator $\mathcal{R}$ of weight zero on a BiHom-Jordan superalgebra $(\mathfrak{J},\cdot,\alpha,\beta)$ is just an $\mathcal{O}$- operator associated to the representation $(\mathfrak{J},ad,\alpha,\beta)$, that is $\mathcal{R}$ commuting with $\alpha$ and $\beta$ and satisfying, for all $x,y\in\mathcal{H}(\mathfrak J)$, the following condition,
 \begin{equation}
  \mathcal{R}(x)\cdot\mathcal{R}(y)=(-1)^{|\mathcal{R}|(|x|+|\mathcal{R}|)}\mathcal{R}\big(\mathcal{R}(x)\cdot y+x\cdot\mathcal{R}(y)\big). \label{RB-op-BiH-Jor-sup}   
 \end{equation}
 \end{re}

 Now, let us consider a representation $(V,\rho,\alpha_V,\beta_V)$ of a BiHom-Jordan superalgebra $(\mathfrak J,\cdot,\alpha,\beta)$ and $T\in\mathcal{O}_{|T|}(\mathfrak J;V,\rho,\alpha_V,\beta_V)$. Define a product $\circ_V$ on the superspace $V$ by 
\begin{equation}\label{Homog-BiH-pre-Jord-sup-V-1}
u\circ_V v=(-1)^{|T|(|u|+|T|)}\rho(T(u))v,\;\forall u,v\in V.    
\end{equation}
This product $\circ_V$ is homogeneous with degree $|T|$ (i.e. $V_i\circ_V V_j\subset V_{i+j+|T|}$ for $i,j\in\mathbb{Z}_2$).
\begin{thm}
Let $(V,\rho,\alpha_V,\beta_V)$ be a representation of a BiHom-Jordan superalgebra $(\mathfrak J,\cdot,\alpha,\beta)$ and $T\in\mathcal{O}_{|T|}(\mathfrak J;V,\rho,\alpha_V,\beta_V)$. Then
\begin{enumerate}
    \item 
If $T$ is even, the product given by Eq. \eqref{Homog-BiH-pre-Jord-sup-V-1} defines on $V$ a BiHom-pre-Jordan superalgebra structure.
\item If $T$ is odd, the product $\bullet_V$ defined on $V$ by
\begin{equation}\label{Homog-BiH-pre-Jord-sup-V-2}
su\bullet_V sv=s(u\circ_V v),    
\end{equation}
defines a BiHom-pre-Jordan superalgebra structure on $sV$.
\end{enumerate}
\end{thm}
\begin{proof}
For any $u,v\in\mathcal{H}(V)$, we have
\begin{align*}
\alpha_V(u\circ_Vv)&=(-1)^{|T|(|u|+|T|)}\alpha_V\rho(T(u))v=\rho(\alpha T(u))\alpha_V(v)\\&=(-1)^{|T|(|u|+|T|)}\rho(T(\alpha_V(u)))\alpha_V(v)=\alpha_V(u)\circ_V\alpha_V(v).
\end{align*}
Similarly, we have $\beta_V(u\circ_Vv)=\beta_V(u)\circ_V\beta_V(v)$.\\
Let $u,v,w,s\in\mathcal{H}(V)$, then by using Eqs. \eqref{cond-O-oper-BiH-Jor1}, \eqref{cond-O-oper-BiH-Jor2} and \eqref{con-BiH-pre-Jord-2}, we have
\begin{align*}
& \sum_{u,v,w}(-1)^{(|T|+|u|)(|T|+|w|)}\big(\alpha_V\beta_V^2(u)\ast_V\alpha_V^2\beta_V(v)\big)\circ_V(\alpha_V^2\beta_V(w)\circ_V\alpha_V^3(s))\\ \quad\quad\quad=&\sum_{u,v,w}(-1)^{(|T|+|u|)(|T|+|w|)}(-1)^{|T|(|T|+|u|+|v|+|w|)}\rho\big( T(\alpha_V\beta_V^2(u)\ast_V\alpha_V^2\beta_V(v))\big)\rho\big(T(\alpha_V^2\beta_V(w))\big)\alpha_V^3(s)\\=&\sum_{u,v,w}(-1)^{|T|(|T|+|u|+|v|+|w|)}(-1)^{(|T|+|u|)(|T|+|w|)}\rho\big( \alpha\beta^2T(u)\cdot\alpha^2\beta T(v))\big)\rho\big(\alpha^2\beta T(w)\big)\alpha_V^3(s)\\\stackrel{\eqref{con-rep-J-2}}{=}& (-1)^{|T|(|T|+|u|+|v|+|w|)}(-1)^{(|T|+|u|)(|T|+|w|)}\rho\big((\beta^2T(u)\cdot\alpha\beta T(v))\cdot\alpha^2\beta T(w)\big)\alpha_V^3\beta_v(s)\\&+(-1)^{|T|(|T|+|u|+|v|+|w|)}(-1)^{(|T|+|u|)(|v|+|w|)}\rho\big(\alpha^2\beta^2T(u)\big)\rho\big(\alpha^2\beta T(w)\big)\rho\big(\alpha^2T(v)\big)\alpha_V^3\beta_V^{-1}(s)\\&+(-1)^{|T|(|T|+|u|+|v|+|w|)}(-1)^{(|T|+|u|)(|v|+|T|)}\rho\big(\alpha^2\beta^2T(v)\big)\rho\big(\alpha^2\beta T(w)\big)\rho\big(\alpha^2T(u)\big)\alpha_V^3\beta_V^{-1}(s)\\=& (-1)^{|T|(|T|+|u|+|v|+|w|)}(-1)^{(|T|+|u|)(|T|+|w|)}\rho\big((T(\beta_V^2(u))\cdot T(\alpha_V\beta_V(v)))\cdot T(\alpha_V^2\beta_V(w))\big)\alpha_V^3\beta_v(s)\\&+(-1)^{|T|(|T|+|u|+|v|+|w|)}(-1)^{(|T|+|u|)(|v|+|w|)}\rho\big(T(\alpha_V^2\beta_V^2(u))\big)\rho\big( T(\alpha_V^2\beta_V(w))\big)\rho\big(T(\alpha_V^2(v))\big)\alpha_V^3\beta_V^{-1}(s)\\&+(-1)^{|T|(|T|+|u|+|v|+|w|)}(-1)^{(|T|+|u|)(|v|+|T|)}\rho\big(T(\alpha_V^2\beta_V^2(v))\big)\rho\big( T(\alpha_V^2\beta_V(w))\big)\rho\big(T(\alpha_V^2(u))\big)\alpha_V^3\beta_V^{-1}(s)\\=&(-1)^{|T|(|T|+|u|+|v|+|w|)}(-1)^{(|T|+|u|)(|T|+|w|)}\rho\Big(T\big((-1)^{|T|(|T|+|u|)}\rho(T(\beta_V^2(u)))\alpha_V\beta_V(v)\\&\quad\quad+(-1)^{|u|(|T|+|v|)}\rho(T(\beta_V^2(v)))\alpha_V\beta_V(u)\big)\cdot T(\alpha_V^2\beta_V(w))\Big)\alpha_V^3\beta_V(s)\\&+(-1)^{|T|(|T|+|u|+|v|+|w|)}(-1)^{(|T|+|u|)(|v|+|w|)}\alpha_V^2\beta_V^2(u)\circ_V\Big(\alpha_V^2\beta_V(w)\circ_V\big(\alpha_V^2(v)\circ_V\alpha_V^3\beta_V^{-1}(s)\big)\Big)\\&+(-1)^{|T|(|T|+|u|+|v|+|w|)}(-1)^{(|T|+|u|)(|T|+|v|)}\alpha_V^2\beta_V^2(v)\circ_V\Big(\alpha_V^2\beta_V(w)\circ_V\big(\alpha_V^2(u)\circ_V\alpha_V^3\beta_V^{-1}(s)\big)\Big)\\=&(-1)^{|T|(|T|+|u|+|v|+|w|)}(-1)^{(|T|+|u|)(|T|+|w|)}\rho\Big(T\big(\beta_V^2(u)\circ_V\alpha_V\beta_V(v)\\&\quad\quad+(-1)^{(|T|+|u|)(|T|+|v|)}\beta_V^2(v)\circ_V\alpha_V\beta_V(u)\big)\cdot T(\alpha_V^2\beta_V(w))\Big)\alpha_V^3\beta_V(s)\\&+(-1)^{|T|(|T|+|u|+|v|+|w|)}(-1)^{(|T|+|u|)(|v|+|w|)}\alpha_V^2\beta_V^2(u)\circ_V\Big(\alpha_V^2\beta_V(w)\circ_V\big(\alpha_V^2(v)\circ_V\alpha_V^3\beta_V^{-1}(s)\big)\Big)\\&+(-1)^{|T|(|T|+|u|+|v|+|w|)}(-1)^{(|T|+|u|)(|T|+|v|)}\alpha_V^2\beta_V^2(v)\circ_V\Big(\alpha_V^2\beta_V(w)\circ_V\big(\alpha_V^2(u)\circ_V\alpha_V^3\beta_V^{-1}(s)\big)\Big)\\=&(-1)^{|T|(|T|+|u|+|v|+|w|)}(-1)^{(|T|+|u|)(|T|+|w|)}\rho\big((\beta_V^2(u)\ast_V\alpha_V\beta_V(v))\ast_V\alpha_V^2\beta_V(w)\big)\alpha_V^3\beta_V(s)\\&+(-1)^{|T|(|T|+|u|+|v|+|w|)}(-1)^{(|T|+|u|)(|v|+|w|)}\alpha_V^2\beta_V^2(u)\circ_V\Big(\alpha_V^2\beta_V(w)\circ_V\big(\alpha_V^2(v)\circ_V\alpha_V^3\beta_V^{-1}(s)\big)\Big)\\&+(-1)^{|T|(|T|+|u|+|v|+|w|)}(-1)^{(|T|+|u|)(|T|+|v|)}\alpha_V^2\beta_V^2(v)\circ_V\Big(\alpha_V^2\beta_V(w)\circ_V\big(\alpha_V^2(u)\circ_V\alpha_V^3\beta_V^{-1}(s)\big)\Big)\\=&(-1)^{|T|(|T|+|u|+|v|+|w|)}(-1)^{(|T|+|u|)(|T|+|w|)}\big((\beta_V^2(u)\ast_V\alpha_V\beta_V(v))\ast_V\alpha_V^2\beta_V(w)\big)\circ_V\alpha_V^3\beta_V(s)\\&+(-1)^{|T|(|T|+|u|+|v|+|w|)}(-1)^{(|T|+|u|)(|v|+|w|)}\alpha_V^2\beta_V^2(u)\circ_V\Big(\alpha_V^2\beta_V(w)\circ_V\big(\alpha_V^2(v)\circ_V\alpha_V^3\beta_V^{-1}(s)\big)\Big)\\&+(-1)^{|T|(|T|+|u|+|v|+|w|)}(-1)^{(|T|+|u|)(|T|+|v|)}\alpha_V^2\beta_V^2(v)\circ_V\Big(\alpha_V^2\beta_V(w)\circ_V\big(\alpha_V^2(u)\circ_V\alpha_V^3\beta_V^{-1}(s)\big)\Big).
\end{align*}
Similarly, we can show that, 
\begin{align*}
&\sum_{u,v,w}(-1)^{(|T|+|u|)(|v|+|w|)}\alpha^2\beta^2(u)\circ_V\big((\alpha\beta(v)\ast_V\alpha^2(w))\circ_V\alpha^3(s)\big)
\\&\quad\quad=\sum_{u,v,w}(-1)^{|T|(|T|+|u|+|v|+|w|)}(-1)^{(|T|+|u|)(|T|+|w|)}\big(\alpha\beta^2(u)\ast_V\alpha^2\beta(v)\big)\circ_V(\alpha^2\beta(w)\circ_V\alpha^3(s)).
\end{align*}
Then, if $T$ is even, it's obvious to conclude that, the binary operation defined by Eq. \eqref{Homog-BiH-pre-Jord-sup-V-1} defines a BiHom-pre-Jordan superalgebra structure on $V$.\\
If $T$ is odd, we have $V_i\circ_VV_j\subseteq V_{i+j+\overline{1}}$. For any $su\in (sV)_i$ and $sv\in (sV)_j$, we have $u\in V_{i+\overline{1}}$ and $v\in V_{j+\overline{1}}$. Therefore, $su\bullet_Vsv=s(u\circ v)\in(sV)_{i+j}$, which gives that $(sV)_i\bullet_V(sV)_j\subseteq (sV)_{i+j}$. Then, for all $su,sv,sw, st\in sV$, we have
\begin{align*}
 & \sum_{u,v,w}(-1)^{|su||sw|}\big(\alpha_V\beta_V^2(su)\star_V\alpha_V^2\beta_V(sv)\big)\bullet_V(\alpha_V^2\beta_V(sw)\bullet_V\alpha_V^3(st))\\\quad\quad\quad=&\sum_{u,v,w}(-1)^{(|u|+1)(|w|+1)}s\big(\alpha_V\beta_V^2(u)\ast_V\alpha_V^2\beta_V(v)\big)\circ_V(\alpha_V^2\beta_V(w)\circ_V\alpha_V^3(t)\big)\\\quad\quad\quad=& (-1)^{(|u|+1)(|w|+1)}s\big((\beta_V^2(u)\ast_V\alpha_V\beta_V(v))\ast_V\alpha_V^2\beta_V(w)\big)\circ_V\alpha_V^3\beta_V(t)\\&+(-1)^{(|u|+1)(|v|+|w|)}s\Big(\alpha_V^2\beta_V^2(u)\circ_V\big(\alpha_V^2\beta_V(w)\circ_V\big(\alpha_V^2(v)\circ_V\alpha_V^3\beta_V^{-1}(t)\big)\big)\Big)\\&+(-1)^{(|u|+1)(|v|+1)}s\Big(\alpha_V^2\beta_V^2(v)\circ_V\big(\alpha_V^2\beta_V(w)\circ_V\big(\alpha_V^2(v)\circ_V\alpha_V^3\beta_V^{-1}(t)\big)\big)\Big) \\\quad\quad\quad=& (-1)^{|su||sw|}(\beta_V^2(su)\star_V\alpha_V\beta_V(sv))\star_V\alpha_V^2\beta_V(sw)\bullet_V\alpha_V^3\beta_V(t)\\&+(-1)^{|su|(|v|+|w|)}\alpha_V^2\beta_V^2(su)\bullet_V\big(\alpha_V^2\beta_V(sw)\bullet_V\big(\alpha_V^2(sv)\circ_V\alpha_V^3\beta_V^{-1}(t)\big)\big)\\&+(-1)^{|su||sv|}\alpha_V^2\beta_V^2(sv)\bullet_V\big(\alpha_V^2\beta_V(w)\bullet_V\big(\alpha_V^2(v)\bullet_V\alpha_V^3\beta_V^{-1}(t)\big)\big).
\end{align*}
By the same way, we can show that
\small{$$\sum_{u,v,w}(-1)^{|su||sw|}\alpha_V^2\beta_V^2(su)\bullet_V\big((\alpha_V\beta_V(sv)\star_V\alpha_V^2(sw))\bullet_V\alpha_V^3(st)\big)
=\sum_{u,v,w}(-1)^{|su||sw|}\big(\alpha_V\beta_V^2(su)\star_V\alpha_V^2\beta_V(sv)\big)\bullet_V(\alpha_V^2\beta_V(sw)\bullet_V\alpha_V^3(st)),$$}
which gives that, the binary operation $\bullet_V$ satisfies conditions \eqref{con-BiH-pre-Jord-2} and \eqref{con-BiH-pre-Jord-3}. Then, we have a BiHom-pre-Jordan superalgebra structure on $sV$ given by $\bullet_V$. The Theorem is proved.
\end{proof}
\begin{cor}
Let $\mathcal{R}$ be a Rota-Baxter operator of weight zero on a BiHom-Jordan superalgebra $(\mathfrak J,\cdot,\alpha,\beta)$. Then     
\begin{enumerate}
\item 
If $\mathcal{R}$ is even, there exists a BiHom-pre-Jordan superalgebra structure on $\mathfrak J$ given by
\begin{equation}\label{BiH-pre-Jor-even-R-B}
x\circ_\mathcal{R}y=\mathcal{R}(x)\cdot y,\;\forall x,y\in\mathfrak J.
\end{equation}
\item If $\mathcal{R}$ is odd, the product $\bullet_\mathcal{R}$ defined on $\mathfrak J$ by
\begin{equation}\label{BiH-pre-Jor-oddn-R-B}
sx\bullet_\mathcal{R} sy=s(x\circ_\mathcal{R} y),\;\forall x,y\in\mathfrak J    
\end{equation}
defines a BiHom-pre-Jordan superalgebra structure on $s\mathcal{J}$.
\end{enumerate}
\end{cor}
\begin{prop-defn}
Let $(V,\rho,\alpha_V,\beta_V)$ be a representation of a BiHom-Jordan superalgebra $(\mathfrak J,\cdot,\alpha,\beta)$. Then, the linear maps $\rho^s:\mathfrak J\to \mathfrak{gl}(sV)$ and $\alpha_V^s,\beta_V^s:sV\to sV$ defined by
\begin{align}
\alpha_V^s(su)=s(\alpha_V(u))\;\;&\text{and}\;\;\beta_V^s(su)=s(\beta_V(u)),\;\forall su\in sV,\label{reverse-repres1}\\
\rho^s(x)su&:=(-1)^{|x|}s(\rho(x)u),\;\forall x\in\mathfrak J,\;su\in sV,\label{reverse-repres2}    
\end{align}
define a representation of $\mathfrak J$ on $sV$, called the parity reverse of the representation $(V,\rho,\alpha_V,\beta_V)$ of $\mathfrak J$.
\end{prop-defn}
\begin{thm}\label{one-one-corr-O-op}
Let $(V,\rho,\alpha_V,\beta_V)$ be a representation of a BiHom-Jordan superalgebra $(\mathfrak J,\cdot,\alpha,\beta)$ and $(sV,\rho^s,\alpha_V^s,\beta_V^s)$ is the parity reverse representation of $\mathfrak J$. Then, there exists a one-one correspondence
between $\mathcal{O}_i(\mathfrak J;V,\rho,\alpha_V,\beta_V)$ and $\mathcal{O}_{i+\overline{1}}(\mathfrak J;sV,\rho^s,\alpha_V^s,\beta_V^s)$ for $i\in\mathbb{Z}_2$.\\

This correspondence between $\mathcal{O}_i(\mathfrak J;V,\rho,\alpha_V,\beta_V)$ and $\mathcal{O}_{i+\overline{1}}(\mathfrak J;sV,\rho^s,\alpha_V^s,\beta_V^s)$ for $i\in\mathbb{Z}_2$ is called the $\mathcal{O}$-operator
duality.
\end{thm}
\begin{proof}
Let us consider a homogeneous linear map $T:V\to \mathfrak J$. Define the following homogeneous linear map $ T^s:sV\to\mathfrak J $ by
\begin{equation}\label{homog-Ts-From-T}
T^s(su):=T(u),\;\forall su\in sV.  
\end{equation}
By the fact that, $s$ is an odd map, we can obviously see that, $|T^s|=|T|+\overline{1}$. Let $x\in\mathfrak J$ and $su\in sV$, then
\begin{align*}
T^s\alpha_V^s(su)&=T^s(s(\alpha_V(u)))\\&= T(\alpha_V(u))\\&=\alpha T(u)=\alpha T^s(su),  
\end{align*}
which gives that, $T^s\alpha_V^s=\alpha T^s$. Similarly, we have $T^s\beta_V^s=\beta T^s$. We have also
$$T^s(\rho^s(x)su)=T^s((-1)^{|x|}s(\rho(x)u))=(-1)^{|x|}T(\rho(x)u).$$
Let $T\in\mathcal{O}_i(\mathfrak J;V,\rho,\alpha_V,\beta_V)$. For all $su,sv\in sV$, we have
\begin{align*}
T^s(su)\cdot T^s(sv)&=T(u)\cdot T(v)\\&=T\big((-1)^{|T|(|T|+|u|)}\rho(T(u))v+(-1)^{|u|(|v|+|T|)}\rho(T(\alpha_V^{-1}\beta_V(v)))\alpha_V\beta_V^{-1}(u)\big)\\&=T^s\big((-1)^{|T|(|T|+|u|)}(-1)^{|T|+|u|}\rho^s(T(u))sv+(-1)^{|u|(|v|+|T|)}(-1)^{|T|+|v|}\rho^s(T(\alpha_V^{-1}\beta_V(v)))s\alpha_V\beta_V^{-1}(u)\big)\\&=T^s\big((-1)^{|T^s|(|T^s|+|su|)}\rho^s(T^s(su))sv+(-1)^{|su|(|sv|+|T^s|)}\rho^s(T^s((\alpha_V^{-1})^s\beta_V^s(sv)))\alpha_V^s(\beta_V^{-1})^s(su)\big).    
\end{align*}
Then $T^s\in\mathcal{O}_{i+\overline{1}}(\mathfrak J;V^s,\rho^s,\alpha_V^s,\beta_V^s)$.\\

Conversely, suppose that $T^s\in\mathcal{O}_{i+\overline{1}}(\mathfrak J;V^s,\rho^s,\alpha_V^s,\beta_V^s)$. Then the homogeneous linear map $T=(T^s)^s:V\to\
 J$ given by Eq. \eqref{homog-Ts-From-T} has degree $i$ and satisfies
 \begin{align*}
 T(u)\cdot T(v)&=(T^s)^s(u)\cdot(T^s)^s(v)=T^s(su)\cdot T^s(sv)\\&=T^s\big((-1)^{|T^s|(|T^s|+|su|)}\rho^s(T^s(su))sv+(-1)^{|su|(|sv|+|T^s|)}\rho^s(T^s((\alpha_V^{-1})^s\beta_V^s(sv)))\alpha_V^s(\beta_V^{-1})^s(su)\big)\\&=T\big((-1)^{|T|(|T|+|u|)}\rho(T(u))v+(-1)^{|u|(|v|+|T|)}\rho(T(\alpha_V^{-1}\beta_V(v)))\alpha_V\beta_V^{-1}(u)\big).
 \end{align*}
  Therefore $T\in\mathcal{O}_i(\mathfrak J;V,\rho,\alpha_V,\beta_V)$.
\end{proof}
\begin{thm}\label{equiv-T-Ts}
Let $(V,\rho,\alpha_V,\beta_V)$ be a representation of a BiHom-Jordan superalgebra $(\mathfrak J,\cdot,\alpha,\beta)$ and $T:V\to\mathfrak J$ be a homogeneous linear map. Then the following assertions are equivalent:
\begin{enumerate}
\item $T\in\mathcal{O}_{|T|}(\mathfrak J;V,\rho,\alpha_V,\beta_V)$;
\item $T^s\in\mathcal{O}_{|T|+\overline{1}}(\mathfrak J;sV,\rho^s,\alpha_V^s,\beta_V^s)$.
\end{enumerate}
\end{thm}
\begin{proof}
Direct computation from Theorem \ref{one-one-corr-O-op}.    
\end{proof}
\begin{defn}
A representation $(V,\rho,\alpha_V,\beta_V)$ of a BiHom-Jordan superalgebra $(\mathfrak J,\cdot,\alpha,\beta)$ is called self-reversing if it is isomorphic to its parity reverse.    
\end{defn}
\begin{prop}
Let $(V,\rho,\alpha_V,\beta_V)$ be a representation of a BiHom-Jordan superalgebra $(\mathfrak J,\cdot,\alpha,\beta)$. Then the direct sum representation $(V\oplus sV,\rho\oplus \rho^s,\alpha_V\oplus \alpha_V^s,\beta_V\oplus\beta_V^s)$ of $\mathfrak J$ is self-reversing.    
\end{prop}
\begin{proof}
 Let $x\in\mathfrak J,\;u\in V$ and $sv\in sV$, we have  
 $$(\rho\oplus\rho^s)^s(x)(su,v)=(-1)^{|x|}\big(s(\rho(x)u),s(\rho^s(x)sv)\big)=(\rho^s(x)su,\rho(x)v)=(\rho^s\oplus\rho)(x)(su,v).$$
 Let $\Phi:V\oplus sV\to sV\oplus V$ defined by
 $$\Phi(u,sv)=(su,v),\;\forall u\in V,sv\in sV.$$
 The map $\Phi$ is clearly bijective and for all $u\in V$ and $sv\in sV$, we have
 \begin{align*}
 \Phi(\alpha_V\oplus\alpha_V^s)(u,sv)&=\Phi(\alpha_V(u),\alpha_V^s(sv))=\Phi(\alpha_V(u),s(\alpha_V(v)))\\&=(s(\alpha_V(u)),\alpha_V(v))=(\alpha_V^s(su),\alpha_V(v)\\&=(\alpha_V\oplus\alpha_V^s)(su,v)=(\alpha_V\oplus\alpha_V^s)\Phi(u,sv),
  \end{align*}
 which implies that $\Phi(\alpha_V\oplus\alpha_V^s)=(\alpha_V\oplus\alpha_V^s)\Phi$. Similarly, we have $\Phi(\beta_V\oplus\beta_V^s)=(\beta_V\oplus\beta_V^s)\Phi$.\\
 Let $x\in\mathfrak J$, we have
 \begin{align*}
 \Phi(\rho\oplus\rho^s)(x)(u,sv)&=\Phi(\rho(x)u,\rho^s(sv))=(s(\rho(x)u),\rho(x)v)\\&=(\rho^s(x)su,\rho(x)v)=(\rho^s\oplus\rho)(x)(su,v)\\&=(\rho^s\oplus\rho)(x)\Phi(u,sv),    
 \end{align*}
 which gives that $\Phi(\rho\oplus\rho^s)(x)=(\rho^s\oplus\rho)(x)(x)\Phi,\;\forall x\in\mathfrak J$.
 Then the direct sum reppresentation $\rho\oplus\rho^s$ is isomorphic to $\rho^s\oplus\rho$ and therefore also isomorphic to $(\rho\oplus\rho^s)^s$. Thus, the representation $(V\oplus sV,\rho\oplus \rho^s,\alpha_V\oplus \alpha_V^s,\beta_V\oplus\beta_V^s)$
 of $\mathfrak J$ is self-reversing. 
\end{proof}

In the following we can extended any $\mathcal{O}$-operator of a BiHom-Jordan superalgebra $(\mathfrak J,\cdot,\alpha,\beta)$ associated to a representation $(V,\rho,\alpha_V,\beta_V)$ to
an $\mathcal{O}$-operator of $(\mathfrak J,\cdot,\alpha,\beta)$ associated to the self-reversing representation given by the following result.
\begin{prop}\label{O-op-T-wid-T}
Let $(V,\rho,\alpha_V,\beta_V)$ be a representation of a BiHom-Jordan superalgebra $(\mathfrak J,\cdot,\alpha_V,\beta_V)$ and $T\in\mathcal{O}_{|T|}(\mathfrak J;V,\rho,\alpha_V,\beta_V)$. Define the homogeneous linear map $\widetilde{T}:V\oplus sV\to\mathfrak J$ by
\begin{equation}\label{O-op-to-O-op-rev}
 \widetilde{T}(u,sv):=T(u),\;\forall u\in V,\;sv\in sV.   
\end{equation}
Then $|\widetilde{T}|=|T|$ and $\widetilde{T}\in\mathcal{O}_{|T|}(\mathfrak J;V,\rho\oplus\rho^s,\alpha_V\oplus\alpha_V^s,\beta_V\oplus\beta_V^s)$.
\end{prop}
\begin{proof}
It's esay to see that $|\widetilde{T}|=|T|$. For any $u\in V$ and $sv\in sV$, by using \eqref{cond-O-oper-BiH-Jor1} we have
\begin{align*}
\widetilde{T}(\alpha_V\oplus\alpha_V^s)(u,sv)&=\widetilde{T}(\alpha_V(u),\alpha_V^s(sv))=\widetilde{T}(\alpha_V(u),s(\alpha_V(sv)))\\&=T(\alpha(u))=\alpha T(u)\\&=\alpha\widetilde{T}(u,sv).
\end{align*}
Then $\widetilde{T}(\alpha_V\oplus\alpha_V^s)=\alpha\widetilde{T}$. Similarly we heve $\widetilde{T}(\beta_V\oplus\beta_V^s)=\beta\widetilde{T}$.\\

Let $u,v\in v$ and $su',sv'\in sV$, we have
\begin{align*}
 \widetilde{T}(u,su')\cdot\widetilde{T}(v,sv')&=T(u)\cdot T(v)\\&= T\big((-1)^{|T|(|T|+|u|)}\rho(T(u))v+(-1)^{|u|(|v|+|T|)}\rho(T(\alpha_V^{-1}\beta_V(v)))\alpha_V\beta_V^{-1}(u)\big)\\&=\widetilde{T}\Big((-1)^{|T|(|T|+|u|)}\rho(T(u))v+(-1)^{|u|(|v|+|T|)}\rho(T(\alpha_V^{-1}\beta_V(v)))\alpha_V\beta_V^{-1}(u),\\&\quad\quad\quad(-1)^{|T|(|T|+|u|)}s(\rho(T(u))v)+(-1)^{|u|(|v|+|T|)}s(\rho(T(\alpha_V^{-1}\beta_V(v)))\alpha_V\beta_V^{-1}(u))\Big)\\&=\widetilde{T}\Big((-1)^{|T|(|T|+|u|)}\rho(T(u))v+(-1)^{|u|(|v|+|T|)}\rho(T(\alpha_V^{-1}\beta_V(v)))\alpha_V\beta_V^{-1}(u),\\&\quad\quad\quad(-1)^{|T|(|T|+|u|)}\rho^s(T(u))sv+(-1)^{|u|(|v|+|T|)}\rho^s(T(\alpha_V^{-1}\beta_V(v))\alpha_V^s(\beta_V^{-1})^s(su))\Big)\\&= \widetilde{T}\Big((-1)^{|T|(|T|+|u|)}(\rho\oplus\rho^s)(T(u))(v,sv)\\&\quad\quad+(-1)^{|T|(|T|+|u|)}(\rho\oplus\rho^s)(T(\alpha_V^{-1}\beta_V(v))(\alpha_V\beta_V^{-1}(u),s(\alpha_V\beta_V^{-1}(u)))\Big)\\&= \widetilde{T}\Big((-1)^{|T|(|T|+|u|)}(\rho\oplus\rho^s)(\widetilde{T}(u,su))(v,sv)\\&\quad\quad+(-1)^{|T|(|T|+|u|)}(\rho\oplus\rho^s)(\widetilde{T}(\alpha_V^{-1}\beta_V(v),s(\alpha_V^{-1}\beta_V(v)))(\alpha_V\beta_V^{-1}(u),s(\alpha_V\beta_V^{-1}(u)))\Big)\\&= \widetilde{T}\Big((-1)^{|\widetilde{T}|(|\widetilde{T}|+|u|)}(\rho\oplus\rho^s)(\widetilde{T}(u,su))(v,sv)\\&\quad\quad+(-1)^{|\widetilde{T}|(|\widetilde{T}|+|u|)}(\rho\oplus\rho^s)\big(\widetilde{T}((\alpha_V\oplus\alpha_V^s)^{-1}(\beta_V\oplus\beta_V^s)(v,sv))\big)(\alpha_V\oplus\alpha_V^s)(\beta_V\oplus\beta_V^s)^{-1}(u,su)\Big).
\end{align*}
Then $\widetilde{T}\in\mathcal{O}_{|T|}(\mathfrak J;V,\rho\oplus\rho^s,\alpha_V\oplus\alpha_V^s,\beta_V\oplus\beta_V^s)$
\end{proof}
\begin{cor}
Let $(V,\rho,\alpha_V,\beta_V)$ be a representation of a BiHom-Jordan superalgebra $(\mathfrak J,\cdot,\alpha,\beta)$ and $T\in\mathcal{O}_{|T|}(\mathfrak J;V,\rho,\alpha_V,\beta_V)$ Then, we have
$$\widetilde{T^s}=(\widetilde{T})^s.$$
\end{cor}
\begin{proof}
By using Theorem \ref{one-one-corr-O-op} and Proposition \ref{O-op-T-wid-T}, we conclude that $\widetilde{T^s},(\widetilde{T})^s\in\mathcal{O}_{|T|+\overline{1}}(\mathfrak J;sV,\rho^s\oplus\rho,\alpha_V^s\oplus\alpha_V,\beta_V^s\oplus\beta_V)$ and for all $su\in sV,v\in V$, we have
$$\widetilde{T^s}(su,v)=T^s(su)=T(u)=\widetilde{T}(u,sv)=(\widetilde{T})^s(su,v),$$
wich gives the Proof.
\end{proof}
\begin{cor}
Let $(V,\rho,\alpha_V,\beta_V)$ be a representation of a BiHom-Jordan superalgebra $(\mathfrak J,\cdot,\alpha,\beta)$ then, we have $\mathcal{O}_{|T|+\overline{1}}(\mathfrak J;V,\rho,\alpha_V,\beta_V)=\{\hat{T}/\hat{T}=T^s\Phi,\;T\in\mathcal{O}_{|T|}(\mathfrak J;V,\rho,\alpha_V,\beta_V)\}$, where $T^s$ is obtained by $T$ via \eqref{homog-Ts-From-T} and  $\Phi:V\to sV$ is an isomorphism between $(V,\rho,\alpha_V,\beta_V)$ and $(sV,\rho^s,\alpha_V^s,\beta_V^s)$.   
\end{cor}
\begin{cor}\label{Ts-By-T-via-isom}
 Let $(V,\rho,\alpha_V,\beta_V)$ be a self-reversing representation of a BiHom-Jordan superalgebra $(\mathfrak J,\cdot,\alpha,\beta)$ and $T:V\to\mathfrak J$ be a homogeneous linear map. Suppose that $\Phi:V\to sV$ is an isomorphism between
the representations $(V,\rho,\alpha_V,\beta_V)$ and $(sV,\rho^s,\alpha_V^s,\beta_V^s)$. Then the following assertions are equivalent. 
\begin{enumerate}
\item $T\in\mathcal{O}_{|T|}(\mathfrak J;V,\rho,\alpha_V,\beta_V)$;
\item $\hat{T}=T^s\Phi\in\mathcal{O}_{|T|+\overline{1}}(\mathfrak J;V,\rho,\alpha_V,\beta_V)$.
\end{enumerate}
\end{cor}
\begin{proof}
It is a direct deduction from Theorem \ref{equiv-T-Ts} and Corollary \ref{Ts-By-T-via-isom}. \end{proof}


\end{document}